\documentclass{article}

\usepackage[shortlabels]{enumitem}
\usepackage{amssymb}
\usepackage{amsmath}
\usepackage{amsthm}
\usepackage{amsopn}
\usepackage{graphicx}
\usepackage{bbm}
\usepackage{mathdots}
\usepackage[export]{adjustbox}
\usepackage{mathdots}
\usepackage{hyperref}
\usepackage{float}
\usepackage{nicefrac}
\usepackage{stmaryrd}

\hypersetup{
    colorlinks = true,
    citecolor=black,
    filecolor=black,
    linkcolor=black,
    urlcolor=black
    linktoc=all,     
    linkcolor=blue,  
}

\graphicspath{{images/}}

\newtheorem{theorem}{Theorem}[section]
\newtheorem{definition}[theorem]{Definition}
\newtheorem{lemma}[theorem]{Lemma}
\newtheorem{remark}[theorem]{Remark}
\newtheorem{example}[theorem]{Example}
\newtheorem{corollary}[theorem]{Corollary}

\newcommand{\CC}{\mathbb C}

\newcommand{\NN}{\mathbb N}

\newcommand{\RR}{\mathbb R}

\newcommand{\ZZ}{\mathbb Z}

\newcommand{\mce}{\mathcal E}

\newcommand{\mch}{\mathcal H}
\newcommand{\mci}{\mathcal I}

\newcommand{\mck}{\mathcal K}
\newcommand{\mcl}{\mathcal L}
\newcommand{\mcm}{\mathcal M}
\newcommand{\mcn}{\mathcal N}
\newcommand{\mco}{\mathcal O}

\newcommand{\mct}{\mathcal T}

\DeclareMathOperator{\dist}{dist}

\DeclareMathOperator{\sgn}{sgn}
\DeclareMathOperator{\supp}{supp}

\DeclareMathOperator{\End}{End}

\DeclareMathOperator{\MS}{MS}
\DeclareMathOperator{\op}{op}
\DeclareMathOperator{\Id}{Id}

\DeclareMathOperator{\inv}{inv}
\DeclareMathOperator{\equi}{eq}
\DeclareMathOperator{\grad}{grad}
\title{Off-diagonal estimates of partial Bergman kernels on $S^1$-symmetric K\"{a}hler manifolds}
\author{Ood Shabtai\thanks{shabtai@imj-prg.fr\newline Partially supported by the DFG funded project SFB/TRR 191 (Project-ID 281071066-TRR 191), and the ANR-DFG project QuaSiDy (Project-ID ANR-21-CE40-0016).}}
\begin{document}
\maketitle
\begin{abstract} We establish local asymptotic estimates of partial Bergman kernels on closed, $S^1$-symmetric K\"{a}hler manifolds. The main result concerns the scaling asymptotics of partial Bergman kernels at generic off-diagonal points in which they are not negligible. The case of the two-dimensional sphere is discussed in detail. \end{abstract}
\tableofcontents
\section{Introduction}
Let $(M, \omega)$ be a connected closed K\"{a}hler manifold of complex dimension $n$. Assume that $L \to M$ is a holomorphic Hermitian line bundle such that the curvature of its Chern connection equals $-i\omega$. The space of holomorphic sections of $L^{\otimes k}$ is a finite-dimensional complex vector space $\mch_k$, equipped with a natural inner product $\langle \cdot, \cdot \rangle$. The latter is obtained by integrating the fiberwise Hermitian product of $L^{\otimes k}$ with respect to the Liouville measure $\mu = \frac{\lvert \omega^{\wedge n} \rvert}{n!}$.

Assume that $\varphi : S^1 \times M \to M$ is a holomorphic Hamiltonian $S^1$-action. Thus, there exists a smooth function $H: M \to \RR$ such that $\varphi$ is the flow of its Hamiltonian vector field $X_H$. The "quantum counterpart" of $H$ is the Hermitian operator $\hat H_k \in \End(\mch_k)$ specified by\footnote{Since $X_H$ generates a holomorphic circle action, $\hat H_k$ preserves holomorphic sections.}
\begin{equation*} \hat H_k = \mcm_H -\frac i k \nabla^{\otimes k}_{X_H},\end{equation*}
where $\nabla : C^\infty(M, L) \to C^\infty(M, T^*M \otimes L)$ is the Chern connection, and $\mcm_H$ is the operator of multiplication by $H$.

Fix a regular value $E \in H(M)$, and consider the spectral projection $$\Pi_{k,E} = \mathbbm 1_{[E, \infty)}(\hat H_k).$$ The Schwartz kernel of $\Pi_{k,E}$, which we also denote by $\Pi_{k,E}$ (by an abuse of notation), is termed a \textit{partial Bergman kernel}\footnote{The projection $\Pi_{k,E}$ is a spectral projection, hence the partial Bergman kernel $\Pi_{k,E}$ is sometimes called a \textit{spectral} partial Bergman kernel.} (\cite{zz1, zz2, berman, coman_mari, pokorny_singer, ross_singer, finski1, finski2}). This terminology references the "full" Bergman kernel, which is the Schwartz kernel of the orthogonal projection $\Pi_k : L^2(M, L^{\otimes k}) \to \mch_k$.
The main goal of the present article is to study the asymptotic behaviour of the partial Bergman kernel $\Pi_{k,E}$  as $k \to \infty$. We focus on the behaviour away from the diagonal $\Delta_M\subset M \times M$, since the asymptotic properties of $\Pi_{k,E}$ on $\Delta_M$ are already understood quite well. Specifically, it is shown (among other things) in \cite{zz1, zz2} that 
\begin{equation}\label{zz_res} \Pi_{k,E}(z,z) = \left\{\begin{array}{ll} \Pi_k(z,z) + \mco(k^{-\infty}) & \text{if } H(z) > E, \\ \Pi_k(z,z)(\frac 1 2 + \mco(k^{-\frac 1 2})) & \text{if } H(z) = E,\\ \mco(k^{-\infty}) & \text{if } H(z) < E.\end{array}\right.\end{equation}
In this work, we formulate statements analogous to (\ref{zz_res}), but for $\Pi_{k,E}(z,w)$, where $z \ne w$.
Notably, $\Pi_{k,E}$ will be shown to "concentrate" about the set
\begin{equation*} \{(z, z) \ | \ H(z) \ge E\} \cup \{(z, \varphi(t,z)) \ | \ H(z) = E,\ t \in S^1\}\subset M \times M.\end{equation*}
It is instructive to compare this type of behaviour with that of the Bergman kernel $\Pi_k$ (or more generally, the kernels of Toeplitz operators\footnote{That is, operators of the form $\Pi_k \mcm_F : \mch_k \to \mch_k$, where $F \in C^\infty(M)$.}). The latter "concenrates" on $\Delta_M$, and additionally, it admits a full, uniform asymptotic expansion. We refer the reader to Sect. \ref{conc_rks} for some further details in this context. The properties of $\Pi_k$ have been studied thoroughly (e.g., \cite{mama, catlin, bbs, dlm, charles, zelditch}), and have found many applications in various fields of mathematics; they also underlie the results presented here.

Perhaps most notably, Theorem 4 of \cite{zz1} describes the leading-order scaling asymptotics of $\Pi_{k,E}(z,z)$ in a $\frac{C}{\sqrt k}$-neighborhood of $H^{-1}(E)$. The main result of the present paper (Theorem \ref{main_thm}) addresses the leading-order scaling asymptotics of $\Pi_{k,E}$ away from the diagonal.

Finally, we refer the reader to \cite{zz1} for a concise review of various properties of $S^1$-symmetric K\"{a}hler manifolds which are relevant in our context. There is some overlap between the present paper and \cite{zz1}; however, to the best of our knowledge, the off-diagonal asymptotic behaviour of partial Bergman kernels corresponding to spectral projections of quantum observables has yet to be described in the research literature. A treatment of the on-diagonal asymptotic properties of partial Bergman kernels of the same type considered here, but without an assumption of $S^1$-symmetry, may be found in \cite{zz2}. 
\subsection{Main results}\label{main_results_sec}
The local behaviour of $\Pi_{k,E}$ away from $H^{-1}(E) \times H^{-1}(E)$ is straightforward to derive (up to an error of order $\mco(k^{-\infty})$) using existing tools from the theory of Toeplitz operators (\cite{charles}), and does not require the assumption that $M$ is equipped with a holomorphic Hamiltonian $S^1$-action. In order to formulate the statements concisely, we use the following terminology (see also Sect. \ref{ms_sec}).
\begin{definition}[\cite{charles}] We say that a sequence $A_k \in \End(\mch_k)$ is negligible at $(z,w) \in M \times M$ if there exists a neighborhood $\mcn \subset M \times M$ of $(z,w)$ such that
\begin{equation*} \sup_{\mcn} \left|A_k \right| = \mco(k^{-\infty}).\end{equation*}
Here (by an abuse of notation), $A_k \in C^\infty\left(M \times M, L^{\otimes k} \boxtimes (L^*)^{\otimes k} \right)$ denotes the Schwartz kernel of $A_k \in \End(\mch_k)$, and its point-wise norm is defined using the Hermitian metric on $L^{\otimes k} \boxtimes (L^*)^{\otimes k}$ which is induced from that of $L$. \end{definition}
\begin{theorem}\label{non_equiv_thm} Let $F \in C^\infty(M)$, and denote $\hat F_k = \Pi_k \left(\mcm_F - \frac i k \nabla^{\otimes k}_{X_F}\right) \Pi_k$. Fix $E \in \left[\min F, \max F\right]$, not necessarily a regular value. Let $z, w\in M$. There exists a neighborhood $\mcn \subset M \times M$ of $(z,w)$ such that the following holds. 
\begin{enumerate}
\item{If $F(z) > E$ or $F(w) > E$, then $\mathbbm 1_{[E, \infty)}(\hat F_k) - \Pi_k = \mathbbm 1_{(-\infty, E)}(\hat F_k)$ is negligible at $(z,w)$.}
\item{If $F(z) < E$ or $F(w) < E$, then $\mathbbm 1_{[E, \infty)}(\hat F_k)$ is negligible at $(z,w)$.} \end{enumerate}
\end{theorem}
\begin{remark} In particular, if $z \ne w$ and $z \not \in F^{-1}(E)$ or $w \not \in F^{-1}(E)$, then $\mathbbm 1_{[E, \infty)}(\hat F_k)$ is negligible at $(z,w)$.\end{remark}
\begin{remark} In some cases (e.g., the settings of Lemma \ref{dist_positive}), the negligibility estimates provided by Theorem \ref{non_equiv_thm} can be improved, using suitable estimates of the Bergman kernel (\cite{mamae}), to exponential decay estimates (cf. \cite{zz1}, Theorem 3). Similarly, in some cases (e.g., the settings of Corollary \ref{dist_positive_equiv}), exponential decay estimates are valid for Schwartz kernels of orthogonal projections onto single eigenspaces (cf. \cite{zz1}, Theorem 1).\end{remark}
We note that Theorem \ref{non_equiv_thm} implies that the ranges of spectral projections corresponding to disjoint classical domains are "orthogonal up to a negligible error". More precisely,
\begin{corollary}\label{two_proj_cor} Let $F, G \in C^\infty(M)$. Assume that $E_1 \in \left(\min F, \max F\right)$, that $E_2 \in \left(\min G, \max G \right)$, and that $\{F \ge E_1\} \cap \{G \ge E_2\} = \emptyset$. Then
\begin{equation*} \Vert \mathbbm 1_{[E_1, \infty)}(\hat F_k) \mathbbm 1_{[E_2, \infty)}(\hat G_k)\Vert_{\op} = \mco(k^{-\infty}).\end{equation*} Consequently, the algebras generated by the pairs $\mathbbm 1_{[E_1, \infty)}(\hat F_k),\ \mathbbm 1_{[E_2, \infty)}(\hat G_k)$ are "asymptotically trivial" (cf. \cite{shabtai2}).\end{corollary}
The main result of the present article describes the scaling asymptotics of $\Pi_{k,E}$ near $z_0, w_0 \in H^{-1}(E)$, $z_0 \ne w_0$. Throughout, we fix the normalization $\min H = 0$ (note that $H$ is defined up to a constant). There exists $N \ge 1$ and an open dense subset $M_N \subset M$ such that for all $z \in M_N$, the stabilizer group of $z$ is the finite subgroup $S^1_N \subset S^1$ of order $N$. Accordingly, the spectrum of $\hat H_k$ equals\footnote{The case $N >1$ is reduced, in Lemma \ref{N>1}, to the case $N = 1$, which is covered by \cite{guillemin_sternberg}.} $H(M) \cap \frac N k \ZZ$. Let $\{\psi_a\}_{a \in \RR}$ denote the gradient flow of $H$, and $\varphi_t(z) = \varphi(t,z)$ (where $t \in S^1 = \RR / 2\pi \ZZ$). 
\begin{theorem}\label{main_thm} Let $\mco_{z_0} \subset M$ denote the $S^1$-orbit of $z_0\in H^{-1}(E)$. Choose a local non-vanishing invariant section $s_{\inv}$ of $L$ in a neighborhood of $\mco_{z_0}$. Denote $$\Pi_{k,E}(z,w) = \mck_{k,E}(z,w) S_k(z,w),$$ where $S_k(z,w) = \sigma^{\otimes k}(z) \otimes (\sigma^*(w))^{\otimes k}$, with $\sigma = \frac{s_{\inv}}{\lvert s_{\inv}\rvert}$.
\begin{enumerate}
\item{If $w_0 \not \in \mco_{z_0}$, then $\Pi_{k,E}$ is negligible at $(z_0,w_0)$.}
\item{If $z_0 \in M_N$ and $t \in S^1$ satisfies $\varphi_{t}(z_0) \ne z_0$, then\begin{multline*} \left(\frac{2\pi} k \right)^n\mck_{k,E}\left(\psi_{\frac a {\sqrt k}}(z_0), \psi_{\frac b {\sqrt k}}\circ \varphi_t(z_0)\right) =\\  e^{-\frac{(a^2+b^2)\Vert X_H(z_0)\Vert^2}2} e^{-iN \lceil\frac{ k E}N \rceil  t}  \frac{N\left( 1-i\cot\left(\frac{N t} 2 \right) \right) } {2\Vert X_H(z_0) \Vert\sqrt{\pi k}} + C_k(a,b,t). \end{multline*}
Here,
\begin{equation*} C_k(a,b,t) = \frac{C_{k,1}(a,b,t)}k + \frac{C_{k,2}(a,b,t)}{k^{\frac 3 2}},\end{equation*}
where $C_{k,1}(a,b,t)$, $C_{k,2}(a,b,t)$ satisfy that for any bounded $B \subset \RR$ and closed interval $\mci \subset S^1 \setminus S^1_N$ there exist $c_{B,\mci,1}, c_{B,\mci,2} > 0$ such that $$|C_{k,1}(a,b,t)| \le (|a|+|b|)c_{B,\mci,1},\ |C_{k,2}(a,b,t)| \le c_{B,\mci,2}$$ for all $a, b \in B,\ t \in \mci$.}
\end{enumerate}
\end{theorem}
Let $V_k(t) = e^{i k t \hat H_k}$ denote the quantum evolution defined by $\hat H_k$. The proof of Theorem \ref{main_thm} is essentially an adaptation of the arguments used in \cite{shabtai}; it relies on basic methods from the theory of Fourier series, applied in the context of the unitary representation $V_k : S^1 \to \End(\mch_k)$. A key role is played by the \textit{periodic Hilbert transform} (\cite{hilbert}), which is a classical singular integral operator on $L^2(S^1)$. We also use an estimate of the Schwartz kernels of projections onto single eigenspaces of $\hat H_k$ (the so-called \textit{equivariant Bergman kernels}). The estimate is obtained by essentially the same method that was used in \cite{zz1} to describe the on-diagonal scaling asymptotics of equivariant Bergman kernels.

\begin{theorem}\label{aux_thm} Let $\lambda_{k}$ be an eigenvalue of $\hat H_k$ such that $\left|\lambda_k - E \right| = \mco(k^{-1})$. Let $\Pi^{\equi}_{k, \lambda_k}$ be the orthogonal projection onto the corresponding eigenspace of $\hat H_k$. Fix $z_0, w_0 \in M$, and denote
\begin{equation*} \Pi_{k,\lambda_k}^{\equi}(z, w) = \mck^{\equi}_{k, \lambda_k}(z,w) S_k(z,w),\end{equation*}
where $S_k(z,w)$ is as specified in Theorem \ref{main_thm}.
\begin{enumerate}
\item{If $H(z_0) \ne E$ or $H(w_0) \ne E$ or $w_0 \not \in \mco_{z_0}$, then $\Pi^{\equi}_{k, \lambda_k}$ is negligible at $(z_0,w_0)$.}
\item{If $z_0 \in M_N$, then
\begin{multline*} \left(\frac{2\pi} k \right)^n \mck^{\equi}_{k, \lambda_k}\left(\psi_{\frac a {\sqrt k}}(z_0),\psi_{\frac b {\sqrt k}}\circ \varphi_t(z_0)\right) = \\  e^{-\frac{(a^2+b^2)\Vert X_H(z_0)\Vert^2}2}e^{-i k\lambda_k  t}\frac N{\Vert X_H(z_0)\Vert\sqrt{\pi k} }+ C^{\equi}_k(a,b,t).\end{multline*}
Here,
\begin{equation*} C^{\equi}_k(a,b,t) = \frac{C^{\equi}_{k,1}(a,b,t)}k + \frac{C^{\equi}_{k,2}(a,b,t)}{k^{\frac 3 2}},\end{equation*}
where $C^{\equi}_{k,1}(a,b,t)$, $C^{\equi}_{k,2}(a,b,t)$ satisfy that for any bounded $B \subset \RR$ and closed interval $\mci \subset S^1$ there exist $c^{\equi}_{B,\mci,1}, c^{\equi}_{B,\mci,2} > 0$ such that $$|C^{\equi}_{k,1}(a,b,t)| \le (|a|+|b|)c^{\equi}_{B,\mci,1},\ |C^{\equi}_{k,2}(a,b,t)| \le c^{\equi}_{B,\mci,2}$$ for all $a, b \in B,\ t \in \mci$.}
\end{enumerate} \end{theorem}
\section{Rotation of the two-dimensional sphere}
In this section, we illustrate the results of Sect. \ref{main_results_sec} in the example of the action of $S^1$, by rotations, on the two-dimensional sphere. The latter is identified (via the stereographic projection from the north pole) with the complex projective line $\CC P^1$ and equipped with the Fubini-Study form such that the total area equals $2\pi$. For conciseness, we focus on the case $N = 1$, and $a= b = 0$.

The Hamiltonian vector field of the function
\begin{equation*} H([z]) = \frac{|z_0|^2}{|z_0|^2+|z_1|^2},\ z = (z_0, z_1) \in \CC^2 \setminus \{0\},\end{equation*}
is specified (on  $\{[z_0 : z_1] \in \CC P^1 \ | \ z_1 \ne 0\}$) by
\begin{equation*} X_H([\zeta : 1]) = i\left(\zeta \partial_\zeta - \bar \zeta \partial_{\bar \zeta} \right), \end{equation*}
and it generates the Hamiltonian flow
\begin{equation*} \varphi_t([z]) = [U_t z],\ U_t = \left(\begin{array}{cc} e^{i t} & 0 \\ 0 & 1 \end{array}\right). \end{equation*}
Clearly, $\varphi : S^1 \times \CC P^1 \to \CC P^1$ is a holomorphic circle action.
\begin{remark} On $S^2 = \{x \in \RR^3 \ | \ |x| = 1\}$, it holds that $$H(x_1,x_2,x_3) = \frac 1 2 (x_3 + 1),$$ and $\varphi_t$ is the rotation by angle $t$ about the $x_3$ axis.\end{remark}
Let $L = \mco(1)$ be the dual of the tautological line bundle on $\CC P^1$. Then $L$ is a prequantum line bundle, and we let $\mch_k = H^0(\CC P^1, L^{\otimes k})$ be the space of holomorphic sections of $L^{\otimes k}$.
The quantum counterpart of $H$ is given by (\ref{corrected_BT})
\begin{equation*} \hat H_k = \Pi_k \mcm_{H - \frac 1 {2k} \Delta H} \Pi_k,\end{equation*}
and in our case, $\Delta H = - 4 H+2$, which means that $$\hat H_k = \frac{k+2}k \Pi_k \mcm_{H} \Pi_k-\frac 1 k \Id_{\mch_k}.$$

As is common, we identify $\mch_k$ with the space of bivariate homogeneous polynomials of degree $k$, denoted by $\CC_k[z_0, z_1]$, and equipped with a suitable inner product so that the set $ \{s_{k,l} \ | \ l = 0,1,...,k\}$, where
\begin{equation}\label{ONB} s_{k,l} = \sqrt{\frac{(k+1)\binom{k}{l}}{2\pi}} z_0^l z_1^{k-l},\end{equation}
forms an orthonormal basis of $\mch_k$.
In fact, this is an eigenbasis of $\hat H_k$ (cf. \cite{lefloch}, Example 5.2.4), with $$\hat H_k(s_{k,l}) =  \lambda_{k,l} s_{k,l},\ \lambda_{k,l} = \frac l k.$$
In particular, $s_{\inv} = s_{1,0}$ is an invariant section of $L \to \CC P^1$.
 
We can estimate the leading order of $s_{k,l}$ using only basic tools, as follows. 
Note that $H([z]) = E$ if and only if $[z] = \left[\sqrt{\frac E {1-E}} e^{i \theta} : 1 \right]$ for some $\theta \in [0,2\pi]$.
\begin{lemma}\label{stirling_est} As in Theorems \ref{main_thm}, \ref{aux_thm}, consider the local normalized section $\sigma = \frac{s_{\inv}}{|s_{\inv}|}$ on $\{[z_0 : z_1] \in \CC P^1\ | \ z_1 \ne 0\}$, and write
\begin{equation*} s_{k,l}([z]) = \kappa_{k,l}([z]) \sigma([z])^{\otimes k}.\end{equation*} Let $E \in (0,1)$ and $\left| \frac{l_k} k - E \right| = \mco(k^{-1})$. Then, applying Stirling's approximation formula, we see that
\begin{equation*}\kappa_{k,l_k}([z]) = \left\{\begin{array}{ll} \mco(k^{-\infty}) & \text{if } H([z]) \ne E,\\ \frac{k^{\frac 1 4}}{(2\pi)^{\frac 3 4}} \frac { e^{i l_k \theta}} {( E(1-E))^{\frac 1 4}}+ \mco(k^{-\frac 3 4}) & \text{if } [z] = \left[\sqrt{\frac E {1-E}} e^{i \theta} : 1\right]. \end{array}\right.\end{equation*}\end{lemma}
\begin{proof}
Stirling's approximation formula produces
\begin{equation*} {k \choose l_k} = \frac 1 {\sqrt{2\pi}} \sqrt{\frac k {l_k(k-l_k)}} \frac{k^k}{l_k^{l_k}(k-l_k)^{k-l_k}}(1 + \mco(k^{-1})).\end{equation*}
Thus, the pointwise norm of $s_{k,l_k}$ is given by
\begin{multline*} \big{|} s_{k,l_k}([\zeta:1])\big{|}_{[\zeta:1]}^2 =\\ \frac{k+1}{(2\pi)^{\frac 3 2}} \sqrt{\frac k {l_k(k-l_k)}} \left(\left(\left(\frac k {l_k} - 1 \right)|\zeta|^2 \right)^{\frac{l_k} k} \frac 1 {1- \frac{l_k} k } \frac 1 {1+|\zeta|^2} \right)^k(1+\mco(k^{-1})).\end{multline*}
Now, note that
\begin{multline*} \left(\left(\frac k {l_k} - 1 \right) |\zeta|^2 \right)^{\frac{l_k} k} \frac 1 {1- \frac{l_k} k } \frac 1 {1+|\zeta|^2} =\\ \left(\left(\frac 1 E - 1 \right) |\zeta|^2\right)^E \frac 1 {1-E} \frac 1 {1+|\zeta|^2} + \mco(k^{-1}).\end{multline*}
We readily verify that
\begin{equation*} a(E, \zeta) = \left(\left(\frac 1 E - 1 \right) |\zeta|^2 \right)^E \frac 1 {1-E} \frac 1 {1+|\zeta|^2} = 1\end{equation*}
if and only if $|\zeta|^2 = \frac{E}{1-E}$, that is, if and only if $H([\zeta:1]) = E$, and otherwise $0 \le a(E,\zeta) < 1$. Thus, $$|\kappa_{k, l_k}([z])| = \big{|} s_{k,l_k}([\zeta:1])\big{|}_{[\zeta:1]} = \mco(k^{-\infty})$$ whenever $H([z]) \ne E$.

Assume that $H([\zeta : 1]) = E$, so that $\zeta = \sqrt{\frac{E}{1-E}} e^{i \theta}$ for some $\theta \in [0,2\pi]$. Then
\begin{equation*}\kappa_{k,l_k}([\zeta:1]) = \sqrt{\frac{k+1}{2\pi}} \sqrt{{k \choose l_k} E^{l_k} (1-E)^{k-l_k}} e^{i l_k \theta} ,\end{equation*}
where
\begin{multline*} {k \choose l_k} E^{l_k} (1-E)^{k-l_k} =\\ \frac 1 {\sqrt{2\pi}} \sqrt{\frac k {l_k(k-l_k)}} \left(\frac {kE} {l_k}  \right)^{l_k} \left(\frac {k(1-E)} {k-l_k} \right)^{k-l_k}\left(1+ \mco(k^{-1})\right),\end{multline*}
and it is readily verified that
\begin{equation*} \sqrt{\frac k {l_k(k-l_k)}} = \frac {1} {\sqrt{k E(1-E)}} + \mco(k^{-\frac 3 2}).\end{equation*}
Writing $\frac {l_k} k = E + \frac{c_k} k$ with $c_k = \mco(1)$, we obtain
\begin{equation*} b_k = \left(\frac k {l_k} E \right)^{l_k} \left(\frac k {k-l_k} (1-E) \right)^{k-l_k} = \left(1 - \frac{c_k}{l_k} \right)^{l_k} \left(1 + \frac{c_k}{k-l_k} \right)^{k- l_k}.\end{equation*}
If $c_k = 0$, then $b_k = 1$. If $c_k \ne 0$, noting that $\left|e - \left(1 + \frac 1 x \right)^x \right| = \mco(x^{-1})$ as $x \to \infty$, we conclude that $b_k = 1 + \mco(k^{-1})$. Thus,
\begin{equation*} {k \choose l_k} E^{l_k}(1-E)^{k-l_k} = \frac 1 {\sqrt{2\pi}} \frac {1} {\sqrt{kE(1-E)}}(1 + \mco(k^{-1})),\end{equation*}
which implies that
\begin{equation*} \kappa_{k,l_k}([z]) = \frac{k^{\frac 1 4}}{(2\pi)^{\frac 3 4}} \frac { e^{i l_k \theta}} {( E(1-E))^{\frac 1 4}}  + \mco(k^{-\frac 3 4}).\end{equation*} \end{proof}
The equivariant Bergman kernel corresponding to the eigenvalue $\frac{l} k$ of $\hat H_k$ is specified by 
\begin{equation}\label{equiv_CP1} \Pi^{\equi}_{k, l}([z], [w]) = s_{k,l}([z]) \otimes s_{k,l}^*([w]).\end{equation} Hence, Lemma \ref{stirling_est} provides an estimate of $\Pi^{\equi}_{k,l}$ which can be used to verify Theorem \ref{aux_thm} (one could similarly verify the cases $N > 1$ or $a \ne 0$ or $b \ne 0$). 
\begin{corollary} As in Theorem \ref{aux_thm}, let $S_k([z],[w]) = \sigma([z])^{\otimes k} \otimes (\sigma^*([w]))^{\otimes k}$, and write $$\Pi^{\equi}_{k,l_k}([z], [w]) = \mck_{k,l_k}^{\equi}([z],[w]) S_k([z],[w]).$$
If we substitute the estimate of Lemma \ref{stirling_est} in (\ref{equiv_CP1}), then
\begin{gather*} \mck^{\equi}_{k,l_k}([z],[w]) =\\ \left\{\begin{array}{ll} \mco(k^{-\infty}) & \text{if } H([z]) \ne E \text{ or } H([w]) \ne E,\\ \frac{k^{\frac 1 2}}{(2\pi)^{\frac 3 2}} \frac{e^{- i l_k t_0}}{\sqrt{E(1-E)}} + \mco(k^{-\frac 1 2}) & \text{if } [z] = \left[\sqrt{\frac E {1-E}} e^{i \theta} : 1 \right],\ [w] = [U_{t_0} z]. \end{array}\right. \end{gather*}
Clearly, the case $H([z]) \ne E$ or $H([w]) \ne E$ is in accordance with Theorem \ref{aux_thm}. Otherwise, $\zeta = \sqrt{\frac E {1-E}} e^{i \theta}$ for some $\theta \in [0,2\pi]$, hence 
\begin{equation*}\Vert X_H([\zeta : 1])\Vert^2 = \frac{2|\zeta|^2}{(1+|\zeta|^2)^2} = 2 E(1-E).\end{equation*} Thus, Theorem \ref{aux_thm} produces the estimate
\begin{equation*} \mck^{\equi}_{k, l_k}([z],[U_{t_0}z]) = \frac k {2\pi} \frac{e^{-i l_k t_0}}{\sqrt{2 E(1-E)} \sqrt{\pi k}}  + \mco(k^{-\frac 1 2}),\end{equation*}
which, as expected, is identical to the estimate coming from Lemma \ref{stirling_est}.\end{corollary}
\begin{figure}[H]

    \begin{minipage}{0.53\textwidth}
        \centering
        \includegraphics[width=1\textwidth]{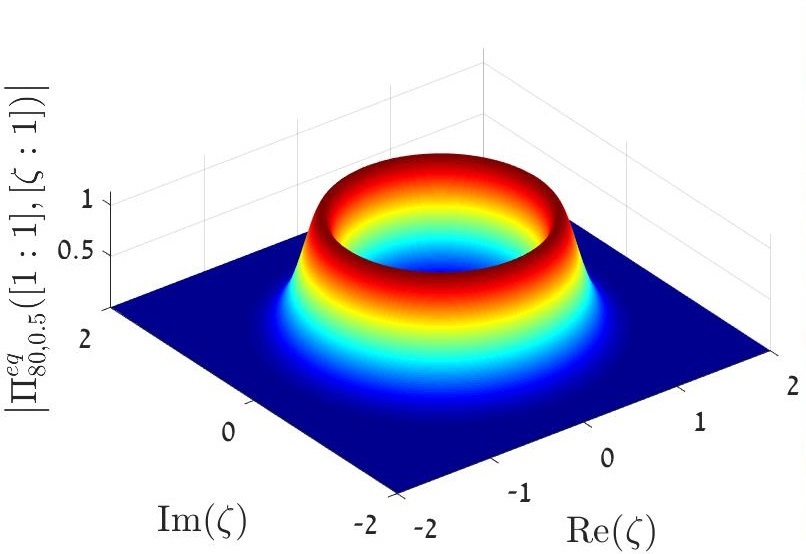} 
    \end{minipage}\hfill
    \begin{minipage}{0.47\textwidth}
        \centering
        \includegraphics[width=0.88\textwidth]{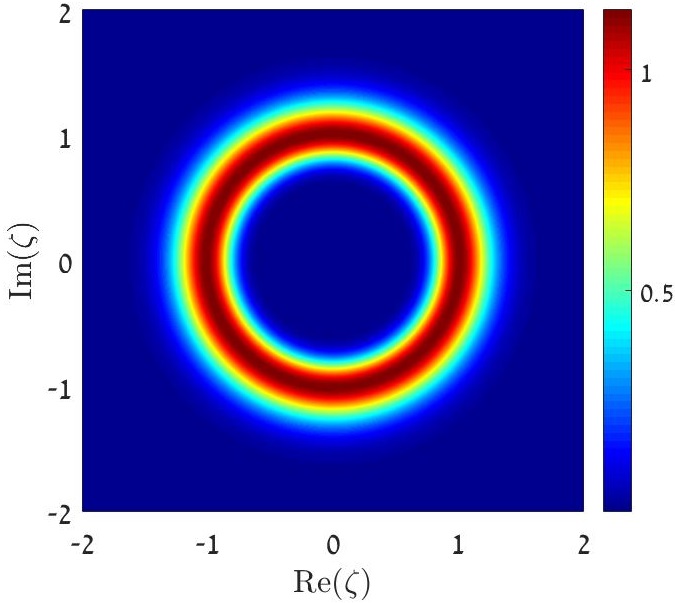} 
    \end{minipage}
    \caption{A plot of $\left|\mck^{\equi}_{80,0.5}([1:1], [\zeta : 1])\right|$. The circle $\{|\zeta| = 1\}=\{H = 0.5\}$ is the orbit of $[1:1]$.}
\end{figure}
The partial Bergman kernel associated with a regular value $E \in (0,1)$ of $H$ is specified by (noting the identification $\mch_k \simeq \CC_k[z_0, z_1]$)
\begin{equation*} \Pi_{k, E} = \frac{k+1}{2\pi} \sum_{\frac l k \ge E} {k \choose l} (z_0 \bar w_0)^l (z_1 \bar w_1)^{k-l}.\end{equation*}
This partial binomial sum is difficult to estimate directly. Denote
\begin{equation*} \Pi_{k,E}([z],[w]) = \mck_{k,E}([z],[w]) S_k([z],[w]).\end{equation*}
The following image illustrates the behaviour of $\mck_{k,E}([z],[w])$ (as expected, the image appears to be in accordance with the estimate provided by Theorem \ref{main_thm}).
\begin{figure}[H]

    \begin{minipage}{0.4\textwidth}
        \centering
        \includegraphics[width=1\textwidth]{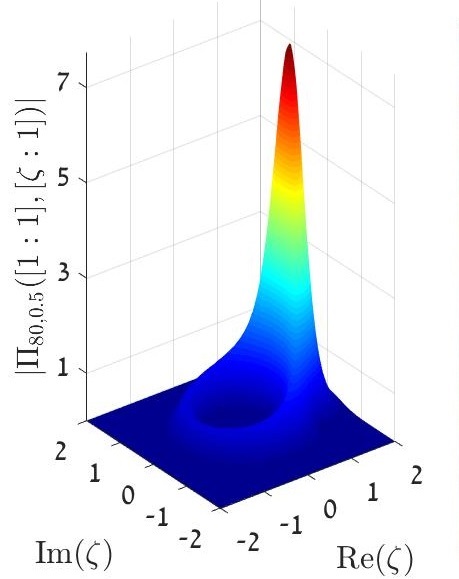} 
    \end{minipage}\hfill
    \begin{minipage}{0.5\textwidth}
        \centering
        \includegraphics[width=1\textwidth]{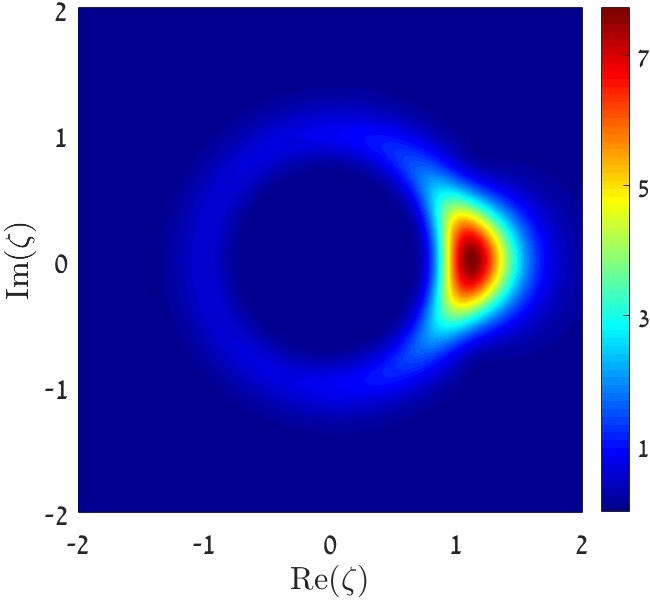} 
    \end{minipage}
    \caption{A plot of $\left|\mck_{80,0.5}([1:1], [\zeta : 1])\right|$. The circle $\{|\zeta| = 1\}=\{H = 0.5\}$ is the orbit of $[1:1]$, and $\zeta = 1$ corresponds to the point $([1:1],[1:1])$ on the diagonal $\Delta_{\CC P^1} \subset \CC P^1 \times \CC P^1$.}
\end{figure}
Consider $[z] \in H^{-1}(E)$, and denote
\begin{equation*} \text{Er}_k(t_0, E, [z]) = \left| \mck_{k, E}([z], [U_{t_0} z]) - \mck_{k, E, \text{approx}}([z], [U_{t_0} z]) \right|,\end{equation*}
where
\begin{equation*} \mck_{k, E, \text{approx}}([z], [U_{t_0}z]) = \frac{k}{4\pi} \frac{e^{-i \lceil k E \rceil t_0}}{\sqrt{2E(1-E)}\sqrt{\pi k}} \left(1-i \cot \left(\frac{t_0} 2 \right) \right) \end{equation*}
is the leading term in the estimate of Theorem \ref{main_thm}. The next images illustrate the behaviour of the error as $k$ grows.
\begin{figure}[H]
    \centering
        \includegraphics[width=1\textwidth]{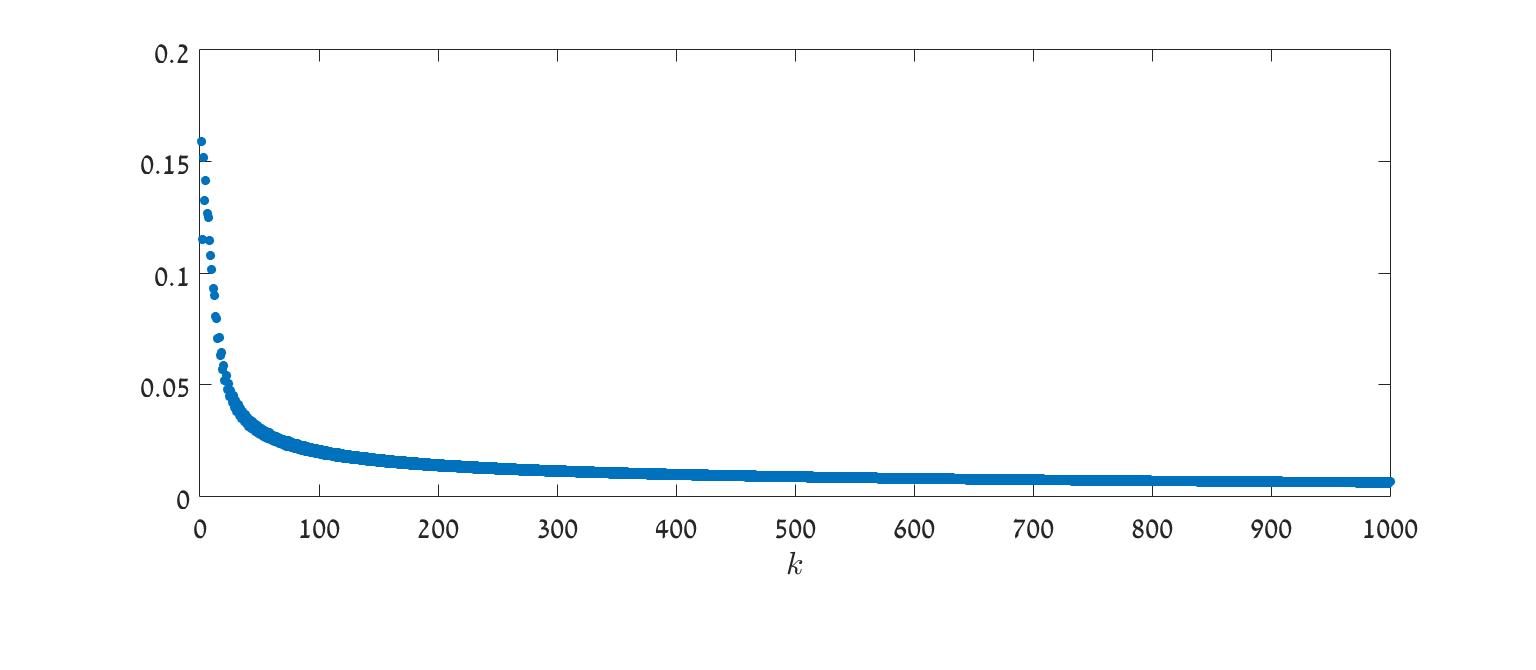}
        \includegraphics[width=1\textwidth]{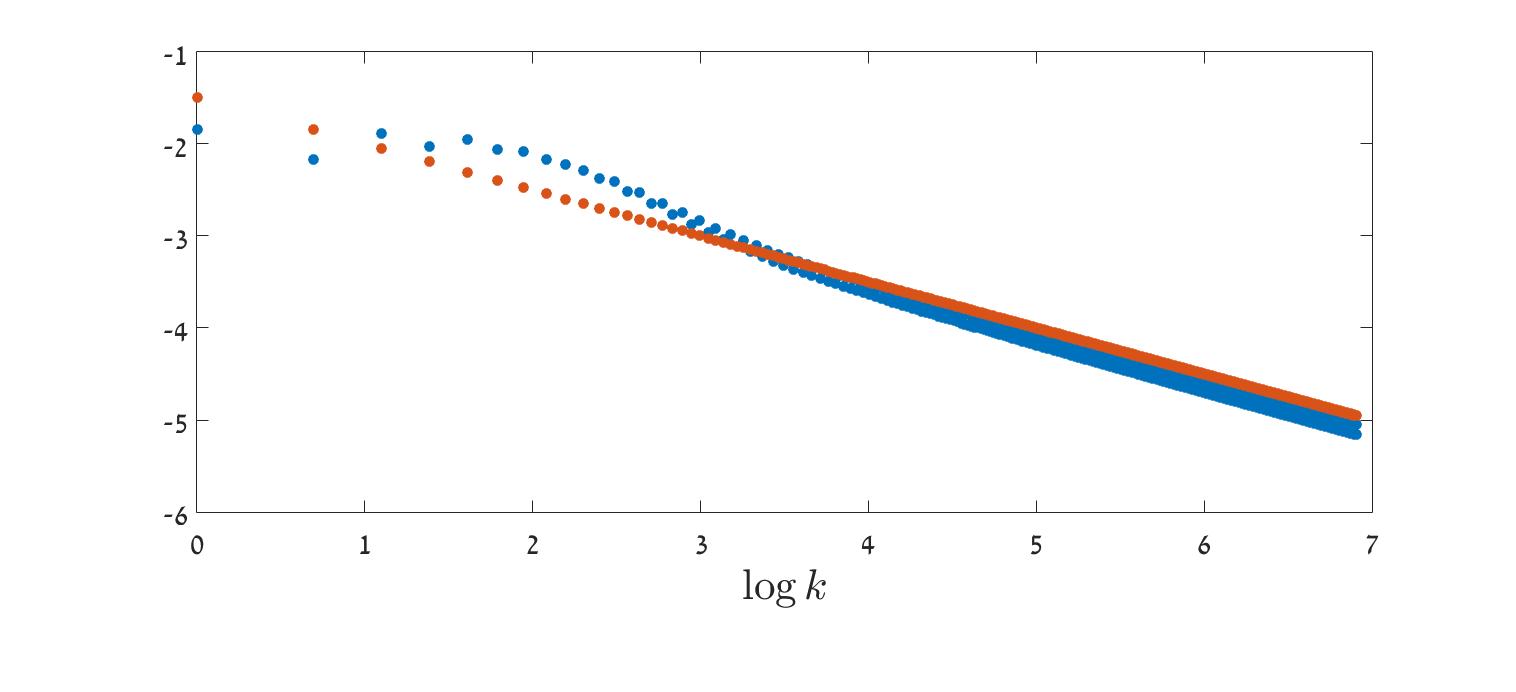}
        \caption{A plot of $\text{Er}_k\left(\frac \pi 2, \frac 1 2, [1:1]\right)$ as a function of $k$ (top), and (bottom) a plot of $\log \text{Er}_k\left(\frac \pi 2, \frac 1 2, [1:1]\right)$ (in blue) and $-1.5 - 0.5 \log(k)$ (in red) as functions of $\log(k)$. The slope of $-0.5$ is in accordance with the fact that the sub-leading term in the estimate provided by Theorem \ref{main_thm} is (in this case) of order $k^{-0.5}$.}\label{fig1}
\end{figure}
\section{The microsupport of partial Bergman kernels}\label{ms_sec}
\subsection{Definition and basic properties}
The notion of microsupport of an admissible sequence of holomorphic sections is specified in \cite{charles}, as follows. Let $(s_k)_{k \in \NN}$ denote a sequence such that $s_k \in \mch_k$. We say that $(s_k)_{k \in \NN}$ is \textit{admissible} if there exists $N > 0$ such that $\Vert s_k \Vert = \mco(k^N)$. An admissible sequence $(s_k)_{k \in \NN}$ is called \textit{negligible at $z_0 \in M$} if there exists a neighborhood $V$ of $z_0$ such that $\sup_{z \in V} |s_k(z)|_z = \mco(k^{-\infty})$, where $|\cdot|_z$ is the norm on $L^{\otimes k}_z$ induced by the Hermitian product.
\begin{definition} The microsupport of an admissible sequence $(s_k)_{k \in \NN}$ is the set $\MS(s_k) \subset M$ specified by
\begin{equation*} \MS(s_k) = M \setminus \{z \in M \ | \ (s_k)_{k \in \NN}\text{ is negligible at }z\}.\end{equation*} \end{definition}
The notion of microsupport is also applicable to sequences $(T_k)_{k \in \NN}$ with $T_k \in \End(\mch_k)$. To this end, equip $M \times M$ with the symplectic form $\pi_1^* \omega - \pi_2^* \omega$, where $\pi_1, \pi_2$ are the natural projections on the left and right factors. Then $\MS(T_k) \subset M \times M$ is defined to be the microsupport of the sequence of Schwartz kernels of $T_k$, which are holomorphic sections of $L^{\otimes k} \boxtimes (L^*)^{\otimes k} \to M \times M$.
\begin{example}(\cite{charles}) Assume that $T_k = \Pi_k \mcm_{G_k} \Pi_k$, where $G_k \in C^\infty(M)$ admits the expansion $G_k = \sum_{l=0}^\infty g_l k^{-l}$ in the $C^\infty$ topology. Then $\MS(T_k) \subset \Delta_M$, where $\Delta_M \subset M \times M$ is the diagonal. Viewed as a subset of $M$,
\begin{equation}\label{ms_formula} \MS(T_k) = \overline{\cup_{l \ge 0} \supp(g_l)}.\end{equation} \end{example}
\subsection{Proof of Theorem \ref{non_equiv_thm}}
The operator $\hat F_k$ admits the representation (\cite{tuynman}, \cite{lefloch}, Proposition 8.1.3)
\begin{equation}\label{corrected_BT} \hat F_k = \Pi_k \mcm_{F - \frac 1 {2k} \Delta F} \Pi_k,\end{equation}
where $\Delta$ is the Laplacian defined by the K\"{a}hler metric of $M$. 
\begin{lemma}\label{smoothed_proj_ms}Let $\varepsilon >0$. Let $\psi_\varepsilon : \RR \to [0,1]$ be a smooth function such that $\psi_\varepsilon(t) = 0$ for all $t \le E-\varepsilon$, $\psi_\varepsilon(t)= 1$ for all $t \ge E$. Write $\hat{\Psi}_{k,\varepsilon} = \psi_\varepsilon\big(\hat{F_k}\big)$. Then $\MS(\hat{\Psi}_{k,\varepsilon}) \subset \{F \ge E-\varepsilon\}$.\end{lemma}
\begin{proof}
There exists (\cite{charles}, Proposition 12) $\Psi_{k,\varepsilon} \in C^\infty(M)$, admitting an expansion $$\Psi_{k,\varepsilon} = \sum_{l=0}^\infty k^{-l} \psi_{l,\varepsilon}$$ in the $C^\infty$ topology, such that $\psi_\varepsilon(\hat F_k )= \Pi_k \mcm_{\Psi_{k,\varepsilon}} \Pi_k + \mco(k^{-\infty})$.
If $F(z) < E - \varepsilon$, then $\psi_\varepsilon$ vanishes in a neighborhood of $F(z)$, hence (using \cite{charles}, p27) there exists a neighborhood $V$ of $z$ such that $\psi_{l,\varepsilon}\big{|}_V \equiv 0$ for all $l \ge 0$. Thus, in light of (\ref{ms_formula}), $z \not \in \MS\big(\psi_\varepsilon(\hat F_k) \big)$.
\end{proof}
The first item of Theorem \ref{non_equiv_thm} now readily follows.
\begin{corollary}\label{ms_cor} The microsupport of $\mathbbm 1_{[E, \infty)}(\hat F_k)$ satisfies $$\MS(\mathbbm 1_{[E, \infty)}(\hat F_k)) \subset \{F \ge E\}\times \{F \ge E\}.$$ The same holds for $\mathbbm 1_{(E, \infty)}(\hat F_k)$. \end{corollary}
\begin{proof}
We prove for $A_k = \mathbbm 1_{[E, \infty)}(\hat F_k)$, but the proof for $\mathbbm 1_{(E, \infty)}(\hat F_k)$ is identical.
Let $\psi_\varepsilon$ be as in Lemma \ref{smoothed_proj_ms}, and write $\hat \Psi_{k,\varepsilon} = \psi_\varepsilon(\hat F_k)$. Then $$A_k = \hat \Psi_{k,\varepsilon} A_k = A_k\hat \Psi_{k,\varepsilon},$$ hence (using \cite{charles}, p24)
\begin{align*} &\MS(A_k) = \MS(\hat \Psi_{k, \varepsilon} A_k)\\ &\subset \{(z_1,z_3) \ | \ \exists z_2 \in M\ \text{such that } (z_1, z_2) \in \MS(\hat \Psi_{k,\varepsilon})\text{ and }(z_2, z_3) \in \MS(A_k)\}\\ &\subset \MS(\hat \Psi_{k,\varepsilon}) \times M,\end{align*}
where by a slight abuse of notation, in the second line $\MS(\hat \Psi_{k, \varepsilon})$ is viewed as a subset of $M \times M$, and in the third line it is viewed as a subset of $M$. Repeating this argument,
\begin{equation*} \MS(A_k) = \MS(A_k\hat \Psi_{k,\varepsilon}) \subset M \times \MS(\hat \Psi_{k,\varepsilon}),\end{equation*}
therefore (noting Lemma \ref{smoothed_proj_ms})
\begin{align*} &\MS(A_k) \subset \left(\MS(\hat \Psi_{k,\varepsilon}) \times M\right) \cap \left( M \times \MS(\hat \Psi_{k,\varepsilon}) \right)\\  &= \MS(\hat \Psi_{k,\varepsilon}) \times \MS(\hat \Psi_{k,\varepsilon}) \subset \{F \ge E -\varepsilon\} \times \{F \ge E-\varepsilon\}.\end{align*}
Since $\varepsilon>0$ is arbitrary, we obtain the required.
\end{proof}
The second item of Theorem \ref{non_equiv_thm} immediately follows from Corollary \ref{ms_cor}.
\begin{corollary} Replacing $F$ with $G = -F$ and $E$ with $-E$, if $G(z) < -E$ or $G(w) < -E$ then there exists a neighborhood $\mcn \subset M \times M$ of $(z,w)$ such that
\begin{equation*} \sup_{\mcn}\left|\mathbbm 1_{(-E, \infty)}(\hat G_k)\right| = \mco(k^{-\infty}),\end{equation*}
where
\begin{gather*} \mathbbm 1_{(-E, \infty)}(\hat G_k) = \Pi_k - \mathbbm 1_{(-\infty, -E]}(\hat G_k) = \Pi_k - \mathbbm 1_{[E, \infty)}(\hat F_k),\end{gather*}
that is, $\sup_{\mcn} \left| \Pi_k - \mathbbm 1_{[E, \infty)}(\hat F_k)\right| = \mco(k^{-\infty})$. \end{corollary}
Finally,
\begin{lemma} Corollary \ref{ms_cor} implies Corollary \ref{two_proj_cor}.\end{lemma}
\begin{proof}
We keep the notations used in the formulation of Corollary \ref{two_proj_cor}, and also write
\begin{equation*} \Pi_{k,1} = \mathbbm 1_{[E_1, \infty)}(\hat F_k),\ \Pi_{k,2}= \mathbbm 1_{[E_2, \infty)}(\hat G_k).\end{equation*}
Then using Corollary \ref{ms_cor},
\begin{align*} &\MS\left(\Pi_{k,1} \Pi_{k,2} \right)\\ &\subset \{(z_1, z_3)\ | \ \exists z_2 \in M \text{ such that } (z_1, z_2) \in \MS(\Pi_{k,1}) \text{ and } (z_2, z_3) \in \MS\left(\Pi_{k,2}\right)\}\\ &\subset \{(z_1, z_3) \ | \ \exists z_2 \in M \text{ such that } z_2 \in \{F \ge E_1\} \cap \{G \ge E_2\}\} = \emptyset.\end{align*}
Thus $\Pi_{k,1} \Pi_{k,2}$ is a negligible sequence, which implies (\cite{charles}, Remark 5) that $\Vert \Pi_{k,1} \Pi_{k,2} \Vert_{\op} = \mco(k^{-\infty})$.\end{proof}
The proof of the first part of the first item of Theorem \ref{aux_thm} is essentially identical to that of Theorem \ref{non_equiv_thm}. Namely,
\begin{lemma}\label{eig_space_proj_MS} Let $\lambda_k$ be a sequence of eigenvalues of $\hat F_k$ with $|\lambda_k - E| = \mco(k^{-1})$, where $E \in F(M)$. Let $\Pi_{k, \lambda_k, F}$ be the orthogonal projection onto the eigenspace associated with $\lambda_k$. Then $\MS(\Pi_{k, \lambda_k, F}) \subset \{F = E\} \times \{F = E\}$. \end{lemma}
\begin{proof}
Given $\varepsilon > 0$ we consider a smooth function $\psi_\varepsilon : \RR \to [0,1]$ such that $\psi_\varepsilon(t) = 1$ whenever $|t - E| < \frac \varepsilon 2$ and $\psi_\varepsilon(t) = 0$ whenever $|t - E| \ge \varepsilon$. Write $\hat \Psi_{k,\varepsilon} = \psi_\varepsilon(\hat F_k)$, and note that exactly as in Lemma \ref{smoothed_proj_ms}, $$\MS(\hat \Psi_{k,\varepsilon}) \subset \{|F - E| \le \varepsilon\} \times \{|F - E| \le \varepsilon\}.$$ For every sufficiently large $k$, it holds that $\hat \Psi_{k,\varepsilon} \Pi_{k, \lambda_k,F} = \Pi_{k,\lambda_k,F} \hat\Psi_{k,\varepsilon} = \Pi_{k,\lambda_k,F}$. Thus, exactly as in Corollary \ref{ms_cor}, we obtain that $$\MS(\Pi_{k, \lambda_k,F}) \subset \MS(\hat\Psi_{k,\varepsilon}) \times \MS(\psi_{\varepsilon}(\hat F_k)) \subset \{|F-E| \le \varepsilon\} \times \{|F - E| \le \varepsilon \}.$$ Since $\varepsilon$ is arbitrary, $\MS(\Pi_{k, \lambda_k}) \subset \{F = E\} \times \{F = E\}$. \end{proof}
\section{Fourier theory and partial Bergman kernels}
\subsection{The Cauchy-Szeg\"{o} projection on the circle}
Let $\RR/2\pi \ZZ = S^1 \subset \CC$ denote the unit circle. Let $\hat g(p) = \langle g, e_p \rangle_{L^2(S^1)}$ denote the $p$-th Fourier coefficient of $g \in L^1(S^1)$, where $p \in \ZZ$ and $e_p(t) = e^{i p t}$.
The \textit{Cauchy-Szeg\"{o} projection} $\Pi_{S^1} :L^2(S^1) \to H^2(S^1)$ is the orthogonal projection on the Hardy space
\begin{equation*} H^2(S^1) = \left\{g \in L^2(S^1) \ | \ \hat g(p) = 0\ \forall p < 0\right\},\end{equation*}
and it admits the formula
\begin{equation*} \Pi_{S^1}(g)(t) = \lim_{r \to 1^-} \frac 1 {2\pi}\int_{-\pi}^\pi \frac{g(t-s)}{1-re^{i s}} ds.\end{equation*}

The \textit{periodic Hilbert transform} $\mch_{S^1} : L^2(S^1) \to L^2(S^1)$ is specified by $\mch_{S^1}(e_p) = -i \sgn(p) e_p$, or equivalently
\begin{equation*} \mch_{S^1}(g)(t) = \lim_{\varepsilon \to 0^+} \frac 1 {2\pi} \int_{\varepsilon \le |t-s| \le \pi} g(s) \cot \left(\frac{t - s} 2 \right) ds.\end{equation*}
\begin{corollary} The Cauchy-Szeg\"{o} projection acts by $e_p(t) \mapsto \mathbbm 1_{[0,\infty)}(p) e_p(t)$, hence it can be expressed in terms of the periodic Hilbert transform, as follows:
\begin{equation*} \Pi_{S^1}(g) = \frac 1 2 \left(i \mch_{S^1}(g) + g + \hat g(0) \right).\end{equation*}\end{corollary}
\subsection{Spectral projections and unitary representations of $S^1$}
Let $\mch$ denote a finite dimensional complex Hilbert space, and write $\dim \mch = d$. Let $A\in \End(\mch)$ be a Hermitian operator. Denote
\begin{equation*} U_A(t) = e^{i t A},\end{equation*}
and assume that $t \mapsto U_A(t)$ is $2\pi$-periodic.

Let $v_1, ..., v_d$ be an orthonormal eigenbasis of $A$ such that $A v_m = p_m v_m$, with $p_1 \le p_2 \le ... \le p_d$. Then $p_m \in \ZZ$, and we can write
\begin{equation*} A = \sum_{m=1}^d p_m v_m \otimes v_m^*,\ U_A(t) = \sum_{m=1}^d e^{i t p_m} v_m \otimes v_m^*,\end{equation*}
where $v_m^*$ is the dual vector of $v_m$.
\begin{corollary} The spectral projection $\Pi_{A} = \mathbbm 1_{[0,\infty)}(A)$ can be written in the following form:
\begin{equation*}  \Pi_{A} = \left(\Pi_{S^1}(U_A) \right)(0) = \lim_{r \to 1^-} \frac 1 {2\pi} \int_{-\pi}^\pi \frac{U_A(s)}{1-re^{is}} ds.\end{equation*}
In terms of the Hilbert transform,
\begin{equation}\label{hilb_tr_repr} \Pi_{A} =  \frac 1 2 \left(i \mch_{S^1}(U_{A})(0) + \Id_{\mch} + \widehat U_{A}(0) \right),\end{equation}
where
\begin{equation*} \widehat U_A(0) = \frac 1 {2\pi} \int_{-\pi}^\pi U_A(t) dt\end{equation*}
and 
\begin{equation*} \mch_{S^1}(U_A)(0) = \frac 1 {2\pi} \int_0^\pi \left(U_{A}(-t) - U_{A}(t)\right) \cot\left(\frac t 2 \right)dt.\end{equation*}\end{corollary}
More generally, given $p,p_0 \in \ZZ$ and $g \in L^2(S^1)$, recall that
\begin{equation*} \widehat{e_{-p_0} g}(p) = \langle e_{-p_0} g, e_p \rangle_{L^2(S^1)} = \langle g, e_{p + p_0} \rangle_{L^2(S^1)} = \hat g(p+p_0).\end{equation*}
This relation allows us to extend formula (\ref{hilb_tr_repr}) to  $\Pi_{A, E} = \mathbbm 1_{[E, \infty)}(A)$, $E \in \RR$.
\begin{corollary} Let $E \in \RR$. Denote $U_{A,E} = e_{-\lceil E \rceil} U_A$. Then
\begin{equation}\label{gen_hilb_tr_repr}\Pi_{A,E} = \left(\Pi_{S^1}(U_{A,E})\right)(0) = \frac 1 2 \left(i \mch_{S^1}(U_{A, E})(0) + \Id_{\mch} + \widehat U_{A, E}(0)\right),\end{equation}
where
\begin{equation*} \widehat U_{A,E}(0) = \frac 1 {2\pi} \int_{-\pi}^\pi U_{A,E}(t) dt = \frac 1 {2\pi} \int_{-\pi}^\pi U_A(t) e^{-i \lceil E \rceil t} dt\end{equation*}
and
\begin{equation*} \mch_{S^1}(U_{A,E})(0) =\frac 1 {2\pi} \int_0^\pi \left(U_A(-t) e^{i \lceil E \rceil t} - U_A(t) e^{-i \lceil E \rceil t} \right) \cot \left(\frac t 2 \right) dt.\end{equation*} \end{corollary}
\subsection{Proof of Theorems \ref{main_thm}, \ref{aux_thm}}
The proof of Theorem \ref{main_thm} relies on formula (\ref{gen_hilb_tr_repr}), which is also applicable to the Schwartz kernels of the operators involved. Thus, we obtain a representation of $\Pi_{k,E}(z,w)$ in terms of the Bergman kernel and the Schwartz kernel of the quantum propagator defined by $\hat H_k$. The well-known asymptotic properties of the Bergman kernel, with the help of the stationary phase lemma, produce the desired estimate. Notably, we can assume without loss of generality that the circle action is effective (Lemma \ref{N>1}), and then the third term in (\ref{gen_hilb_tr_repr}) is the equivariant Bergman projection associated with the eigenvalue $\frac{\lceil k E \rceil}k$ of $\hat H_k$. Accordingly, the proof of Theorem \ref{main_thm} relies on the second item of Theorem \ref{aux_thm}. The latter is established by similar arguments, namely, the stationary phase lemma is used together with the integral representation of equivariant Bergman kernels as "Fourier coefficients" of the full Bergman kernel.

Recall that there exists $N \ge 1$ and an open dense subset $M_N \subset M$ such that the stabilizer group of every $z \in M_N$ is the subgroup $S^1_N \subset S^1$ of order $N$. We assume throughout that $N = 1$ (i.e., that the circle action is effective); this is justified by the following lemma.
\begin{lemma}\label{N>1} It suffices to prove Theorems \ref{main_thm}, \ref{aux_thm} in the case $N = 1$. \end{lemma}
\begin{proof}
Assume that $N > 1$. If $t \in S^1_N$, then $\varphi_t|_{M_N} = \Id$, hence by continuity $\varphi_t = \Id$. It follows that $t \mapsto \varphi_{N,t} = \varphi_{\frac t N}$, $t \in \RR/2\pi\ZZ$, is an effective holomorphic circle action, generated by the Hamiltonian $H_N = \frac 1 N H$. Writing $\hat H_{N,k} = \frac 1 N \hat H_k$,
\begin{equation*} \Pi_{k, E} = \mathbbm 1_{[E, \infty)}(\hat H_k) = \mathbbm 1_{[\frac E N, \infty)}\left(\hat H_{N,k}\right).\end{equation*}
Thus, the case $N > 1$ is reduced to the case $N = 1$.
Similarly,
\begin{equation*} \Pi^{\equi}_{k, \lambda_k} = \mathbbm 1_{\{\lambda_k\}}(\hat H_k) = \mathbbm 1_{\{\frac{\lambda_k} N\}}(\hat H_{N,k}).\end{equation*}
\end{proof}
Recall that $E \in H(M)$ is a regular value of $H$, and assume that $z \in H^{-1}(E)$. Denote $V_k(t) = e^{i k t \hat H_k}$, and consider the map $$V_k(z,w) : S^1 \to L_z^{\otimes k} \otimes (L_w^*)^{\otimes k}$$ specified by $t \mapsto V_k(t)(z,w)$, where the latter is the Schwartz kernel of $V_k(t)$.
Denote $E_k = \frac{\lceil k E \rceil} k$ and $V_{k,E_k}(t)(z,w) = e^{- i kE_k t} V_k(t)(z,w)$.
According to (\ref{gen_hilb_tr_repr}),
\begin{equation}\label{hilb_tr_form_ker} \Pi_{k,E}(z,w) = \frac 1 2 \left(i \mch_{S^1}\left(V_{k,E_k}(z,w)\right)(0) + \Pi_k(z,w) + \widehat{V_{k,E_k}(z,w)}(0)\right),\end{equation}
where \begin{equation*} \widehat{V_{k,E_k}(z,w)}(0) = \frac 1 {2\pi} \int_0^{2\pi} V_{k,E_k}(t)(z,w) dt = \Pi^{\equi}_{k, E_k}(z,w)\end{equation*} is the equivariant Bergman kernel associated with the eigenvalue $E_k$ of $\hat H_k$.
The Schwartz kernel of $V_{k,E_k}(t)$ admits the following useful formula.
\begin{lemma}\label{shift_prop_ker} The Schwartz kernel of $V_{k,E_k}(t)$ is specified by
\begin{equation*} V_{k,E_k}(t)(z,w) = c_{k,E_k}(z,t) \mct_{k,t}\left(\Pi_k(\varphi_t(z), w)\right),\end{equation*}
where $\mct_{k,t} : L^{\otimes k}_{\varphi_t(z)} \otimes (L^*)^{\otimes k}_w \to L^{\otimes k}_z \otimes (L^*)^{\otimes k}_w$ is the parallel transport along the curve $t' \mapsto \left(\varphi_{t-t'}(z), w \right)$, $t' \in [0,t]$, and $c_{k,E_k}(z,t) = e^{i k(H(z) - E_k)t}= e^{ik(E- E_k)t}$.\end{lemma}
\begin{proof}
As shown in \cite{lefloch}, Proposition 8.2.1, the action of $V_k(t)$ is specified by
\begin{equation*} V_k(t)s(z) = e^{i k t H(z)} \tau_{k,t} s(\varphi_t(z)),\ s \in \mch_k,\end{equation*}
where $\tau_{k,t} : L^{\otimes k}_{\varphi_t(z)} \to L^{\otimes k}_z$ is the parallel transport along the curve $t' \mapsto \varphi_{t-t'}(z)$, $t' \in [0,t]$. Thus, considering any orthonormal basis $s_1,...,s_{d_k}$ of $\mch_k$, we find that the Schwartz kernel of $V_k(t)$ is given by
\begin{align*} V_k(t)(z,w) &= \sum_{m = 1}^{d_k} \left(V_k(t) s_m\right)(z) \otimes s_m^*(w)\\ &= e^{i k t H(z)} \sum_{m=1}^{d_k} \left(\tau_{k,t} s_m \right)(\varphi_t(z)) \otimes s_m^*(w)\\ &= e^{i k t H(z)} \left(\tau_{k,t} \otimes \Id\right) \left(\sum_{m=1}^{d_k} s_m(\varphi_t(z)) \otimes s_m^*(w) \right) \\ &= e^{ik t H(z)} \mct_{k,t} \left(\Pi_k(\varphi_t(z), w) \right).\end{align*}
Multiplying both sides by $e^{- i kE_k t}$, we obtain the required. \end{proof}
The microsupport $\MS(\Pi_k)$ of $\Pi_k$ equals the diagonal $\Delta_M \subset M \times M$. In particular, this implies (see \cite{charles}, Proposition 8, or \cite{mamae}, Theorem 1) that for every $\varepsilon > 0$ and $N > 0$ there exists $C_{\varepsilon, N} > 0$ such that for every $z, w \in M$, if $\dist(z,w) \ge \varepsilon$ then 
\begin{equation}\label{berg_sup_est}  \lvert \Pi_k(z,w) \rvert < C_{\varepsilon, N} k^{-N},\end{equation}
and that for every vector field $Y$ on $M \times M$ there exists $C_{Y, \varepsilon, N}$ such that for every $z, w \in M$, if $\dist(z,w) \ge \varepsilon$ then
\begin{equation}\label{berg_der_sup_est} \lvert (\tilde \nabla_k)_Y \Pi_k(z,w) \rvert < C_{Y, \varepsilon, N} k^{-N}.\end{equation}
Here, $\tilde \nabla_k$ is the connection on $L^{\otimes k} \boxtimes (L^*)^{\otimes k}$ induced by $\nabla$. 
The first part of Theorem \ref{main_thm} readily follows from these estimates.
\begin{lemma}\label{dist_positive} If $z \in M$ and $w \in M \setminus \mco_z$, then $(z, w) \not \in \MS(\Pi_{k,E})$. \end{lemma}
\begin{proof}
Let $\mcn_0 \subset M \times M$ denote a neighborhood of $(z,w)$ such that for every $(z',w') \in \mcn_0$ and $t \in [-\pi, \pi]$,
$$\dist(\varphi_t(z'),w') > \frac \varepsilon 2$$
for some $\varepsilon >0$. We wish to prove that $$\sup_{  \mcn_0} \lvert \Pi_{k,E} \rvert = \mco(k^{-\infty}),$$ which would mean (by definition) that $(z, w) \not \in \MS(\Pi_{k,E})$. 
%
Let $$\mcn = \{(\varphi_t(z'), w') \ | \ (z',w') \in \mcn_0,\ t \in [-\pi, \pi]\}.$$ 
Looking at the representation (\ref{hilb_tr_form_ker}) in light of (\ref{berg_sup_est}), we immediately note that the second term satisfies
\begin{equation*} \sup_{\mcn_0}|\Pi_k| = \mco(k^{-\infty}).\end{equation*}
Similarly, estimate (\ref{berg_sup_est}) implies that the term $\widehat{V_{k,E_k}(z',w')}(0) = \Pi^{\equi}_{k,E_k}(z',w')$, which is the equivariant Bergman kernel associated with the eigenvalue $E_k$, satisfies
\begin{equation}\label{equiv_neglig} \sup_{(z', w') \in \mcn_0} \left|\Pi^{\equi}_{k,E_k}(z',w') \right| \le \sup_{\mcn_0}\frac 1 {2\pi} \int_{-\pi}^\pi \left| \Pi_k\left(\varphi_t(z'), w' \right) \right| dt = \mco(k^{-\infty}).\end{equation}
Finally, consider the first term in (\ref{hilb_tr_form_ker}). Let $V'_{k,E_k}(u)(z',w') = \frac{d}{dt}\big{(} V_k(z',w')\big{)}(u)$. Note that $V_{k,E_k}(0)(z',w') = \Pi_k(z',w')$. If $t \in (0,\pi)$, then by the mean-value theorem,
\begin{equation*} |V_{k,E_k}(t)(z',w') - \Pi_k(z',w')| \le t \sup_{u \in [0,t]} \left|V_{k,E_k}'(u)(z',w')\right|,\end{equation*}
where (in light of Lemma \ref{shift_prop_ker})
\begin{multline*} V'_{k,E_k}(u)(z',w')=\\  c_{k,E_k}(u) \mct_{k,u}\left(ik (E- E_k)\Pi_k(\varphi_u(z'), w') +  (\tilde \nabla_k)_{Y_H}(\Pi_k)(\varphi_u(z'), w')\right),\end{multline*}
with $Y_H = (X_H, 0)$.

Applying (\ref{berg_sup_est}) and (\ref{berg_der_sup_est}), we see that
\begin{equation}\label{der_est}\begin{aligned} \left|V'_{k,E_k}(u)(z',w')\right| &\le k|E - E_k| |\Pi_k(\varphi_u(z'), w') | + |(\tilde \nabla_k)_{Y_H}\Pi_k(\varphi_u(z'), w')|\\  &\le \sup_{\mcn}|\Pi_k| + \sup_{\mcn}|(\tilde \nabla_k)_{Y_H} \Pi_k|= \mco(k^{-\infty}).\end{aligned} \end{equation}
Thus, 
\begin{equation*} \sup_{\mcn_0 \times (0,\pi)} \frac 1 t \left|V_{k,E_k}(t)(z',w')- \Pi_k(z',w')\right| = \mco(k^{-\infty}),\end{equation*}
which implies (since if $t \approx 0$, then $\cot\left(\frac t 2 \right) \approx \frac 2 t $) that
\begin{equation*} \sup_{\mcn_0} \left|\mch_{S^1}(V_{k,E_k}(z',w'))(0) \right| = \mco(k^{-\infty}).\end{equation*}
\end{proof}
We proceed with the proof of the second item of Theorem \ref{main_thm}, relying on formula (\ref{hilb_tr_form_ker}). Fix $a, b \in \RR$. The behaviour of the third term in (\ref{hilb_tr_form_ker}) is specified in Theorem \ref{aux_thm}, hence in what follows we focus on $\mch_{S^1}\left(V_{k,E_k}(z_{a_k,0},z_{b_k,t_0})\right)(0)$, where we denote $a_k = \frac a {\sqrt k}$, $b_k = \frac b {\sqrt k}$, and for any $\alpha\in \RR$, $t \in S^1$,
\begin{equation*} z_\alpha = \psi_\alpha(z_0),\ z_{\alpha,t} = \varphi_t (z_\alpha).\end{equation*}
\begin{lemma}\label{on_orbit} Fix $z_0 \in H^{-1}(E) \cap M_1$. Let $\mci \subset S^1 \setminus \{0\}$ be a closed interval. Then for every fixed $\delta > 0$ small enough, if $t_0 \in \mci$ then
\begin{gather*} \mch_{S^1}(V_{k,E_k}(z_{a_k}, z_{b_k,t_0}))(0) =\\- \frac 1 {2\pi} \int_{t_0 - \delta}^{t_0 + \delta} V_{k, E_k}(t)(z_{a_k}, z_{b_k,t_0}) \cot \left(\frac t 2 \right) dt + r_k(a,b,t_0).\end{gather*}Here, $r_k(a,b,t_0) = \mco(k^{-\infty})$ uniformly, in the following sense: for every bounded set $B \subset \RR$ and $N > 0$, there exists $c_{B,N} > 0$ such that $|r_k(a,b,t_0)| \le c_{B,N} k^{-N}$ for every $a, b \in B$, $t_0 \in \mci$.\end{lemma}
\begin{proof}
Assume that $\mci = [t', t''] \subset (0,2\pi)$. Let $\delta>0$ be small enough so that
\begin{equation*}\left( [0, \delta] \cup [2\pi - \delta, 2\pi]\right) \cap [t'-\delta, t''+ \delta] = \emptyset.\end{equation*}%
There exists $\varepsilon > 0$ such that for every $t \in [0,2\pi]$ and $t_0 \in \mci$, if $|t-t_0| \ge \delta$ then $\dist(z_{0,t}, z_{0, t_0}) > \varepsilon$.
Let $B \subset \RR$ be bounded. There exists $k_0 \in \NN$ such that for every $k \ge k_0$, for every $a, b \in B$ and for every $t \in [0,2\pi]$, $t_0 \in \mci$ such that $|t - t_0| \ge \delta$, it holds that \begin{equation}\label{dist_big_eps}\dist(z_{a_k, t}, z_{b_k, t_0}) > \frac \varepsilon 2.\end{equation}
Thus, the mean-value theorem together with a suitable version of estimate (\ref{der_est}) imply that
\begin{equation*} \sup_{(0, \delta]} \left| \left(V_{k,E_k}(t)(z_{a_k}, z_{b_k,t_0})- V_{k,E_k}(-t)(z_{a_k}, z_{b_k,t_0}) \right) \cot \left(\frac t 2 \right) \right| = \mco(k^{-\infty}),\end{equation*}
and (in light of (\ref{dist_big_eps}), (\ref{berg_der_sup_est})) the estimate is uniform for $a,b \in B$ and $t_0 \in \mci$. Thus,
\begin{multline*} \int_0^\pi \left(V_{k, E_k}(t)(z_{a_k}, z_{b_k,t_0})- V_{k,E_k}(-t)(z_{a_k}, z_{b_k,t_0}) \right) \cot \left(\frac t 2 \right)dt =\\ \int_\delta^\pi \left(V_{k,E_k}(t)(z_{a_k},z_{b_k,t_0}) - V_{k,E_k}(-t)(z_{a_k},z_{b_k,t_0}) \right) \cot \left(\frac t 2 \right)dt +\mco(k^{-\infty})\\ = \int_\delta^{2\pi - \delta} V_{k,E_k}(t)(z_{a_k},z_{b_k,t_0}) \cot \left(\frac t 2 \right)dt + \mco(k^{-\infty}),\end{multline*}
where (in light of (\ref{dist_big_eps}), (\ref{berg_sup_est})) the remainder estimate is uniform for $a,b \in B$ and $t_0 \in \mci$. Finally,
\begin{equation*} \int_\delta^{t_0 - \delta} V_{k,E_k}(t)(z_{a_k},z_{b_k,t_0}) \cot \left(\frac t 2 \right)dt = \mco(k^{-\infty}),\end{equation*}
and again, by (\ref{dist_big_eps}), (\ref{berg_sup_est}), the estimate is uniform for $a, b \in B$, $t_0 \in \mci$. Similarly for the integration over $[t_0 + \delta, 2\pi - \delta]$.
\end{proof}
\begin{corollary}\label{hilb_trans_corollary} Let $S_{k}(z,w) = \sigma(z)^{\otimes k} \otimes (\sigma^*_{\inv}(w))^{\otimes k}$, where $\sigma = \frac{s_{\inv}}{|s_{\inv}|}$, with $s_{\inv}$ a local non-vanishing invariant section of $L$ in a neighborhood of $\mco_{z_0}$. Note that $\Vert X_H\Vert$ is constant in $\mco_{z_0}$. In light of the previous lemma and (\ref{stat_phase_prop_eq}), we conclude that
\begin{gather*} -\left(\frac{2\pi} k \right)^n\mch_{S^1}(V_{k,E_k}(z_{a_k},z_{b_k,t_0}))(0) =\\  \left(e^{-\frac{(a^2 + b^2)}2 \Vert X_H(z_0)\Vert^2}e^{-ik E_kt_0 } \frac {\cot \left(\frac{t_0} 2 \right)} {\Vert X_H(z_0) \Vert\sqrt{\pi k}} +C_k(a, b,t_0)\right)S_k(z_{a_k},z_{b_k,t_0}).\end{gather*}
Here, $C_k( a, b,t_0) = \frac{C^v_k( a, b,t_0)}{2\pi}$, where $C_k^v(a, b,t)$ is as specified in Corollary \ref{stationary_phase_propagator}, with $v(t) =  \cot \left(\frac t 2 \right)$.\end{corollary}

In order to complete the proof of Theorem \ref{main_thm}, it remains to establish the off-diagonal estimate for the equivariant Bergman kernel $\Pi_{k,E_k}$, as formulated in Theorem \ref{aux_thm}. Hence, we now turn to address the proof of the latter.
Assume that $\lambda_k$ is an eigenvalue of $\hat H_k$ such that $\lvert\lambda_k - E\rvert = \mco(k^{-1})$. Let
\begin{equation*} \Pi^{\equi}_{k, \lambda_k} = \frac 1 {2\pi} \int_0^{2\pi} V_k(t) e^{- ik\lambda_k t} dt\end{equation*}
denote the orthogonal projection on the eigenspace associated with $\lambda_k$.
\begin{corollary}\label{dist_positive_equiv} Estimate (\ref{equiv_neglig}) in the proof of Lemma \ref{dist_positive} is also valid for $\Pi^{\equi}_{k, \lambda_k}$. Namely, if $z\in M$ and $w \in M \setminus \mco_z$, then $(z,w) \not \in \MS(\Pi^{\equi}_{k,\lambda_k})$. Together with Lemma \ref{eig_space_proj_MS}, we obtain the first item of Theorem \ref{aux_thm}. \end{corollary}
\begin{corollary} Let $z_0 \in H^{-1}(E)\cap M_1$ and $\mci \subset S^1$ be a closed interval. The same argument as in the proof of Lemma \ref{on_orbit} implies that for every fixed $\delta > 0$ small enough, if $t_0 \in \mci$ then
\begin{equation*} \Pi^{\equi}_{k, \lambda_k}(z_{a_k},z_{b_k,t_0}) = \int_{t_0 - \delta}^{t_0 + \delta} V_{k}(t)(z_{a_k},z_{b_k,t_0}) e^{-i k \lambda_k t} dt + r^{\equi}_k(a, b,t_0).\end{equation*}
Here, $r^{\equi}_k(a,b,t_0) = \mco(k^{-\infty})$ uniformly, as before: for every $B \subset \RR$ bounded and $N > 0$, there exists $c_{B,N} > 0$ such that $|r^{\equi}_k(a,b,t_0)| \le c_{B,N} k^{-N}$ for every $a, b \in B$, $t_0 \in \mci$.

Thus, by (\ref{stat_phase_prop_eq}),
\begin{gather*} \left(\frac{2\pi} k \right)^n\Pi^{\equi}_{k, \lambda_k}(z_{a_k},z_{b_k,t_0}) = \\ \left(e^{-\frac{(a^2 + b^2)}2 \Vert X_H (z_0) \Vert^2}e^{-i  k \lambda_k t_0}\frac {1}{\Vert X_H(z_0) \Vert\sqrt{\pi k} } + C^{\equi}_k( a, b,t_0)\right)S_k(z_{a_k},z_{b_k,t_0}),\end{gather*}
where $C_k^{\equi}(a, b,t_0) = \frac{C^v_k(a, b, t_0)}{2\pi}$, with $C_k^v(a, b,t)$ as specified in Corollary \ref{stationary_phase_propagator}, for $v = 1$ (also, $S_k(z,w)$ is as specified in Corollary \ref{hilb_trans_corollary}).\end{corollary}
\section{Estimates of integrals involving the Bergman kernel}
We begin by noting that $\Pi_k$ admits the representation (\cite{charles}, \cite{lefloch}, Theorem 7.2.1, \cite{hsiaomari}, Theorem 4.11)
\begin{equation}\label{charles_form} \Pi_k(z,w) = \left(\frac k {2\pi} \right)^n u(z,w,k) \mce^{\otimes k}(z,w) + R_k(z,w),\end{equation}
where $\mce \in C^\infty(M \times M, L \boxtimes L^*)$, $u(\cdot, \cdot, k) \in C^\infty(M \times M)$ is real valued, and $R_k$ have the following properties.
\begin{itemize}

\item{The section $\mce$ satisfies $|\mce(z,w)| < 1$ for $z \ne w$, and\footnote{Here, we use the Hermitian metric to make the identification $L_z \otimes L^*_z \simeq \CC$.} $\mce(z,z) = 1$.}

\item{The function $u(\cdot, \cdot,k)$ is real-valued, and $u(\cdot, \cdot, k) \sim \sum_{l =0}^\infty k^{-l} u_l(\cdot, \cdot)$ in the $C^\infty$-topology\footnote{This means that for every $N \ge 0$, the function $u(\cdot, \cdot, k) - \sum_{l=0}^N k^{-l} u_l(\cdot, \cdot)$ and all its derivatives are uniformly $\mco(k^{-(N+1)})$.}, with $u_0(z, z) \equiv 1$.}

\item{$R_k = \mco(k^{-\infty})$ uniformly in $(z,w)$.}

\end{itemize}
Recall that $\{\varphi_t\}_{t \in [0,2\pi]}$ and $\{\psi_a\}_{a \in \RR}$ denote the Hamiltonian flow and gradient flow associated with $H$. The two flows commute, and define a $\CC^*$-action on $M$ (the gradient flow also consists of biholomorphisms). As before, given $z_0 \in M$, if $\alpha \in \RR$ and $t \in S^1$, then we denote $$z_\alpha = \psi_\alpha(z_0),\ z_{\alpha, t} =  \varphi_t(z_\alpha).$$
\begin{lemma}\label{non_deg_lemma} Fix $z_0 \in M$. Let $s_{\inv}$ be a a local invariant holomorphic section of $L$ in a neighborhood of $\mco_{z_0}$. Let $S(z,w) = \sigma(z) \otimes \sigma^*(w)$, where $\sigma = \frac{s_{\inv}}{\lvert s_{\inv}\rvert}$. For $\delta > 0$ small enough, define $$g_{z_0} : (-\delta, \delta)\times (-\delta, \delta) \times (-\delta, \delta)\times S^1 \to \CC$$ by
\begin{equation}\label{E_parallel_transport}  \mce(z_{a,t+t'}, z_{b,t'}) = e^{ig_{z_0}(t,a,b,t')} S(z_{a,t+t'},z_{b,t'}).\end{equation} Let $t_0 \in S^1$. Then
\begin{align*} &g_{z_0}(0,t_0) = 0,\\ &\partial_{t} g_{z_0}(0,t_0) = H(z_0), \partial_a g_{z_0}(0,t_0) = \partial_b g_{z_0}(0,t_0) = 0,\\ &\partial_{t}^2 g_{z_0}(0,t_0) =  \partial_a^2 g_{z_0}(0,t_0) = \partial_b^2 g_{z_0}(0,t_0) = - \partial_a \partial_b g_{z_0}(0,t_0) = \frac i 2 \Vert X_H \Vert^2,\\ &\partial_a \partial_{t} g_{z_0}(0,t_0) = \partial_b \partial_{t} g_{z_0}(0,t_0) = \frac 1 2 \Vert X_H \Vert^2.\end{align*}\end{lemma}
\begin{proof}
First, $\mce(\varphi_{t_0}(z_0), \varphi_{t_0}(z_0))  = S(\varphi_{t_0}(z_0),\varphi_{t_0}(z_0))$, hence $g_{z_0}(0,t_0) = 0$.

Let $\tilde \nabla$ denote the connection on $L \boxtimes L^* \to M \times M$ induced from $\nabla$. Then
\begin{multline*}  \tilde{\nabla}_{(X_H, 0)} \mce (z_{a,t+t_0}, z_{b,t_0}) =\\ i\partial_t g_{z_0}(t,a,b,t_0)\mce(z_{a,t+t_0}, z_{b,t_0}) \\+ e^{i g_{z_0}(t, a, b,t_0)} \tilde \nabla_{(X_H, 0)} S(z_{a,t+t_0}, z_{b,t_0}) \\=i\left( \partial_{t} g_{z_0}(t,a, b,t_0)- H(z_a)\right) \mce(z_{a,t+t_0}, z_{b,t_0}),\end{multline*}
where we used the invariance of $s_{\inv}$, and the fact that $\varphi_{t_0}$ preserves level sets of $H$.

Let $\alpha_\mce$ be the 1-form defined in a neighborhood of the diagonal $\Delta_M \subset M \times M$ by the equation \begin{equation}\label{alpha_E_def}\tilde \nabla \mce = -i\alpha_\mce \otimes \mce.\end{equation}
Then
\begin{equation*} i\left( \partial_t g_{z_0}(t, a, b,t_0)- H(z_a)\right)  = - i \alpha_\mce(X_H,0)(z_{a,t+t_0}, z_{b,t_0}),\end{equation*}
that is,
\begin{equation*} \partial_t g_{z_0}(t, a, b,t_0) = H(z_a) - \alpha_{\mce}(X_H, 0)(z_{a,t+t_0}, z_{b,t_0}).\end{equation*}
However, $\alpha_\mce$ vanishes on $\Delta_M$ (\cite{lefloch}, Lemma 7.1.3), hence $\partial_tg_{z_0}(0,t_0) = H(z_0)$. Similarly (noting that $\nabla_{\grad H} \sigma = 0$),
\begin{gather*} \partial_a g_{z_0}(t, a, b,t_0) = - \alpha_{\mce}(\grad H, 0)(z_{a,t+t_0}, z_{b,t_0}),\\ \partial_b g_{z_0}(t,a,b,t_0) = -\alpha_{\mce}(0, \grad H)(z_{a,t+t_0}, z_{b,t_0}),\end{gather*}
and since $\alpha_{\mce}$ vanishes on $\Delta_M$, we obtain $\partial_a g_{z_0}(0,t_0) = \partial_b g_{z_0}(0,t_0) = 0$.

Next, we note that
\begin{align*} &\partial_t^2 g_{z_0} = -\mcl_{(X_H, 0)}\alpha_{\mce}(X_H, 0),\ \partial_a^2 g_{z_0} = - \mcl_{(\grad H, 0)} \alpha_{\mce}(\grad H, 0),\\ &\partial_b^2 g_{z_0} = - \mcl_{(0, \grad H)} \alpha_{\mce}(0, \grad H), \partial_a \partial_b g_{z_0} = - \mcl_{(\grad H, 0)} \alpha_{\mce}(0, \grad H),\\ & \partial_t \partial_a g_{z_0} = - \mcl_{(X_H, 0)}\alpha_{\mce}(\grad H, 0),\ \partial_t \partial_b g_{z_0} = -\mcl_{(X_H, 0)} \alpha_{\mce}(0, \grad H).\end{align*}
Now, as shown in \cite{lefloch}, Lemma 7.1.3, there exists a smooth section $B_{\mce}$ of the bundle $T^*(M \times M) \otimes T^*(M \times M) \otimes \CC \to \Delta_M$ such that for vector fields $X,Y$ on $M \times M$,
\begin{equation*} \mcl_{X}(\alpha_\mce(Y)) = B_\mce(Y, Y) = \tilde \omega(q(Y_H), Y_H)\end{equation*}
on $\Delta_M$. Here, $q$ is the projection from $T_{(w,w)}(M \times M) \otimes \CC$ onto $T_{(w,w)}^{0,1}(M \times M)$ with kernel $T_{(w,w)}\Delta_M \otimes \CC$, and $\tilde \omega = \pi_1^* \omega - \pi_2^* \omega$ is the K\"{a}hler form on $M \times M$. We can compute $q(X,0)$ and $q(0, X)$ explicitly, and obtain
\begin{equation*} q(X,0) = \frac 1 2 \left(X +i j X, -X + ij X\right),\ q(0, X) = \frac 1 2 \left(-X + ij X, X + ij X \right)\end{equation*}
where $j$ is the complex structure on $M$. Thus, a straightforward computation produces
\begin{align*} \partial^2_t g_{z_0}(0,t_0) = - \tilde \omega(q(X_H, 0), (X_H, 0)) = - \frac 1 2 \omega(i j X_H, X_H) = \frac i 2 \Vert X_H \Vert^2,\end{align*}
and noting that $\grad H = - j X_H$,
\begin{equation*} \partial_t \partial_a g_{z_0}(0,t_0) = - \tilde \omega(q(X_H, 0), (\grad H, 0)) = - \frac 1 2 \omega(X_H, -j X_H) = \frac 1 2 \Vert X_H \Vert.\end{equation*}
The rest of the cases are computed in the same way.
\end{proof}
Using Lemma \ref{non_deg_lemma}, we can apply the stationary phase approximation in order to estimate integrals (along $S^1$-trajectories) which involve the Bergman kernel. In light of Lemma \ref{N>1}, we assume that $N = 1$ (i.e., there exists an open dense $M_1 \subset M$ such that the stabilizer group of every $z \in M_1$ is trivial).
\begin{corollary}\label{stationary_phase_propagator} Let $\lambda_k$ be an eigenvalue of $\hat H_k$ such that $|\lambda_k - E| = \mco(k^{-1})$, where $E \in H(M)$ is a regular value. Fix $z_0 \in H^{-1}(E) \cap M_1$. Let $\mci \subset S^1$ be a closed interval. Let $v$ be a smooth function, compactly supported in a neighborhood of $\mci$, such that $v|_{\mci} \ne 0$. Fix $a, b \in \RR$. For $\delta > 0$ small enough and $t_0 \in \mci$, denote
\begin{equation*} I_{k, \delta}^{v,\lambda_k}(\alpha, \beta,t_0) =\left(\frac{2\pi} k \right)^n \int_{t_0 - \delta}^{t_0 + \delta} v(t) V_{k, \lambda_k}(t)(z_{\alpha},z_{\beta, t_0}) dt,\end{equation*}
where $V_{k,\lambda_k}(t) = e^{-ik \lambda_k t} V_k(t)$. Write $a_k = \frac a {\sqrt k}$, $b_k = \frac b {\sqrt k}$. Then
\begin{equation}\label{stat_phase_prop_eq}\begin{gathered} I^{v,\lambda_k}_{k, \delta}\left(a_k,b_k,t_0\right) = \mck^{v,\lambda_k}_{k}(a, b,t_0) S_k(z_{a_k},z_{b_k, t_0}),\end{gathered}\end{equation}
with $S_k$ as specified in Theorem \ref{main_thm}, and
\begin{equation*} \begin{gathered}\mck^{v,\lambda_k}_{k}(a, b,t_0) = \\e^{-\frac{(a^2 + b^2)\Vert X_H(z_0)\Vert^2}2 }e^{-i k\lambda_k  t_0}  \frac{2 v(t_0)}{\Vert X_H(z_0) \Vert} \sqrt{\frac \pi k} + C_k^v(a, b,t_0).\end{gathered} \end{equation*}
Here, 
\begin{equation*} C^v_k(a,b,t) = \frac{C^v_{k,1}(a,b,t)}k + \frac{C^v_{k,2}(a,b,t)}{k^{\frac 3 2}},\end{equation*}
where $C^v_{k,1}(a,b,t)$, $C^v_{k,2}(a,b,t)$ satisfy that for any bounded set $B \subset \RR$ there exist $c^v_{B,1}, c^v_{B,2} > 0$ such that $$|C^v_{k,1}(a,b,t)| \le (|a|+|b|)c^v_{B,1},\ |C^v_{k,2}(a,b,t)| \le c^v_{B,2}$$ for all $a, b \in B$ and $t \in \mci$.\end{corollary}
\begin{proof}
According to Lemma \ref{shift_prop_ker},
\begin{equation*} V_{k,\lambda_k}(t)(z_{\alpha},z_{\beta, t_0}) = c_{k,\lambda_k}(z_{\alpha},t) \mct_{k,t}\Pi_k(z_{\alpha, t}, z_{\beta, t_0}),\end{equation*}
where $c_{k,\lambda_k}(z,t) = e^{ik(H(z) - \lambda_k) t}$. Then using (\ref{charles_form}),
\begin{gather*} \mct_{k,t} \Pi_k(z_{\alpha, t}, z_{\beta, t_0})=\\ \left(\frac k {2\pi} \right)^n u(z_{\alpha, t}, z_{\beta, t_0}, k) \mct_{k,t}\left(\mce^{\otimes k}(z_{\alpha, t},z_{\beta, t_0})\right).\end{gather*}
The invariance of $s_{\inv}$ implies that $\mct_{k,t} S_k(\varphi_t(z), w) = e^{-ik H(z)t} S_k(z,w)$. Thus, in the notations of Lemma \ref{non_deg_lemma},
\begin{equation*} \mct_{k,t} \left(\mce^{\otimes k}(z_{\alpha, t},z_{\beta, t_0} )\right) = e^{ik( g_{z_0}(t-t_0,\alpha,\beta,t_0)-H(z_{\alpha})t)} S_k(z_{\alpha}, z_{\beta, t_0}).\end{equation*}
Hence,
\begin{equation*}  I^{v,\lambda_k}_{k, \delta}(a_k,b_k,t_0) =\left(\int_{t_0 - \delta}^{t_0 + \delta} \tilde v_k(t, a_k, b_k,t_0)e^{ikg_{z_0}(t-t_0, a_k,b_k,t_0)} dt\right) S_k(z_{a_k},z_{b_k, t_0}),\end{equation*}
with
\begin{equation*} \tilde v_k(t, \alpha, \beta,t_0) = v(t)e^{-i k \lambda_k t} u(z_{\alpha, t},z_{\beta, t_0}, k) .\end{equation*}
Next, applying a change of variables,
\begin{gather*} \int_{t_0 - \delta}^{t_0 + \delta} \tilde v_k(t,a_k,b_k,t_0)e^{ikg_{z_0}(t-t_0,a_k,b_k,t_0)} dt \\= e^{-i k \lambda_k t_0} \int_{-\delta}^\delta  v_k(t, a_k, b_k,t_0) e^{i k f_{z_0}(t,a_k, b_k,t_0)} dt,\end{gather*}
where (noting that $H(z_0) = E$)
\begin{align*} &v_k(t,\alpha,\beta,t_0) = v(t+t_0) e^{i k(E - \lambda_k)t} u(z_{\alpha, t+t_0}, z_{\beta, t_0},k),\\ &f_{z_0}(t, \alpha, \beta,t') = g_{z_0}(t, \alpha, \beta,t') - E t.\end{align*}

Now, $\partial_t f_{z_0}(0,t_0) = 0$ and $\partial_t^2 f_{z_0}(0,t_0)\ne 0$ by Lemma \ref{non_deg_lemma}, and we can assume without loss of generality that $t=0$ is the unique such point in $(-\delta,\delta)$. Also, the derivatives of $v_k$ are bounded as $k \to \infty$ (since $|E - \lambda_k| = \mco(k^{-1})$). Hence, we can apply the stationary phase lemma for a complex valued phase (\cite{hormander}, Theorem 7.7.12), and obtain
\begin{equation*}\begin{aligned}&\int_{-\delta}^\delta v_k(t, a_k, b_k,t_0) e^{ik f_{z_0}(t,a_k, b_k,t_0)}dt \\&= \left(\frac{2 \pi i}{k \left(\partial_t^2 f_{z_0}\right)^0(a_k, b_k,t_0)}\right)^{\frac 1 2}e^{i kf_{z_0}^0(a_k, b_k,t_0)} v_k^0(a_k, b_k,t_0) + \mco(k^{-\frac 3 2}).  \end{aligned}\end{equation*}
Here, for a function $F(t,a,b,t')$, the notation $F^0(a,b,t')$ stands for a function of $a,b,t'$ only which belongs to the same residue class as $F$ modulo the ideal generated by $\partial_{t} F$ (i.e., $F - F^0= G \partial_{t}F$ for some function $G$). Also, the square root is defined such that its real part is non-negative. Finally, the $\mco(k^{-\frac 3 2})$ estimate is uniform for $a,b$ in bounded sets and $t_0 \in \mci$.

The derivatives of $v_k^0$ are bounded as $k \to \infty$, hence (using Taylor's theorem)
\begin{align*} v_k^0(a_k, b_k,t_0) &= v_k^0(0,t_0) + \frac{c_0(k,a,b,t_0)}{\sqrt k} \\ &= v_k(0,t_0) +  \frac{c_0(k,a,b,t_0)}{\sqrt k} \\ &= v(t_0) + \frac{c_0(k,a,b,t_0)}{\sqrt k}+\mco(k^{-1})\end{align*}
for some $c_0(k,a,b,t_0)$ satisfying that for every bounded set $B \subset \RR$ there exists $c_{0,B} > 0$ such that $|c_0(k,a,b,t_0)| \le (|a|+|b|) c_{0,B}$ for all $a,b \in B$, $t_0 \in \mci$. Also, the $\mco(k^{-1})$ estimate is uniform for $a, b \in B$ and $t_0 \in \mci$.

Similarly,
\begin{align*}  &\left(\frac{2 \pi i}{k \left(\partial_t^2 f_{z_0}\right)^0(a_k, b_k,t_0)}\right)^{\frac 1 2} =  \left(\frac{2 \pi i}{k \left(\partial_t^2 f_{z_0}\right)^0(0,t_0)}\right)^{\frac 1 2} +  \frac{c_1(k,a,b,t_0)}{k}\\ &= \frac 2 {\Vert X_H\Vert} \sqrt{\frac \pi k} +  \frac{c_1(k,a,b,t_0)}{ k}.\end{align*}
Again, $c_1(k,a,b,t_0)$ satisfies that for every bounded set $B \subset \RR$ there exists $c_{1,B} > 0$ such that $|c_1(k,a,b,t_0)| \le (|a|+|b|) c_{1,B}$ for all $a,b \in B$, $t_0 \in \mci$. 

Finally, we may choose (see \cite{hormander}, 7.7.16, and the succeeding paragraph)
\begin{equation*} f^0_{z_0}(a,b,t') = f_{z_0}(0,a,b,t')-q(0,a,b,t')T(a,b,t')^2,\end{equation*}
where $T$ such that $T(0,  t_0) = 0$ for every $t_0 \in \mci$, and $q$ is some smooth function. It follows that
\begin{align*} &f^0_{z_0}(0,t_0) = f_{z_0}(0,t_0) = 0,\\ &\partial_a f^0_{z_0}(0,t_0) = \partial_a f_{z_0}(0,t_0) = 0 = \partial_b f_{z_0}(0,t_0) = \partial_b f^0_{z_0}(0,t_0).\end{align*}
Also,
\begin{align*} &\partial_a^2 f^0_{z_0}(0,t_0) = \partial_a^2 f_{z_0}(0,t_0) -2 q(0,t_0) (\partial_a T(0,t_0))^2,\\ &\partial_a \partial_b f^0_{z_0}(0,t_0) = \partial_a\partial_b f_{z_0}(0,t_0) -2 q(0,t_0) \partial_b T(0,t_0)\partial_a T(0,t_0),\\ &\partial_b^2 f^0_{z_0}(0,t_0) = \partial_b^2 f_{z_0}(0,t_0) -2 q(0,t_0) (\partial_b T(0,t_0))^2,\end{align*}
with
\begin{align*} &2q(0,t_0) = \partial_t^2 f_{z_0}(0,t_0),\\ &\partial_a T(0,t_0) = -\frac{\partial_a \partial_t f_{z_0}(0,t_0)}{\partial_t^2 f_{z_0}(0,t_0)},\ \partial_bT(0,t_0) = -\frac{\partial_b \partial_t f_{z_0}(0,t_0)}{\partial_t^2 f_{z_0}(0,t_0)}.\end{align*} Thus, noting Lemma \ref{non_deg_lemma}, a straightforward computation produces
\begin{align*} f^0_{z_0}(a_k,b_k,t_0) &= f^0_{z_0}(0,t_0) + a_k \partial_a f^0_{z_0}(0, t_0) + b_k \partial_b f^0_{z_0}(0,t_0)\\ &+ \frac{a_k^2}2 \partial_a^2 f^0_{z_0}(0,t_0) + a_k b_k \partial_a \partial_b f^0_{z_0}(0,t_0)\\ &+ \frac{b_k^2}2 \partial_b^2 f^0_{z_0}(0,t_0) + \frac{c_2(k,a,b,t_0)}{k^{\frac 3 2}}\\ &= \frac i {2k}(a^2 + b^2) \Vert X_H \Vert^2 + \frac{c_2(k,a,b,t_0)}{k^{\frac 3 2}}.\end{align*}
As before, $c_2(k,a,b,t_0)$ satisfies that for every bounded set $B \subset \RR$ there exists $c_{2,B} > 0$ such that $|c_2(k,a,b,t_0)| \le (|a|+|b|)c_{2,B}$ for all $a, b \in B$, $t_0 \in \mci$. Putting everything together, we obtain the required.
\end{proof}
\section{Concluding remarks}\label{conc_rks}
The spectral projections addressed in Theorem \ref{main_thm} (the main result of this work) are analogues of so-called \textit{Melrose-Uhlmann projections} \cite{guillemin_lerman}, which are defined in the framework of pseudodifferential quantization. Melrose-Uhlmann projections on $S^1$-symmetric manifolds are closely related to the quantization of symplectic cuts (\cite{lerman, guillemin_lerman}). More generally, they are instances of operators whose Schwartz kernels are \textit{distributions of Melrose-Uhlmann type} (\cite{melrose_uhlmann}); these are distributions associated with (suitable) pairs of intersecting Lagrangian submanifolds, and they admit a symbol calculus. It would be interesting to study whether an analogous theory exists in the framework of geometric quantization.

Finally, we note that our study of partial Bergman kernels has two obvious shortcomings. First, the main result (Theorem \ref{main_thm}) only addresses the case of Hamiltonians generating holomorphic circle actions; we except that more robust methods would make it possible to establish a version of Theorem \ref{main_thm} which is valid in greater generality (cf. \cite{zz2}). Second, the results presented in this work are all local; it would be desirable to obtain a uniform description of partial Bergman kernels (it is not obvious that such a description exists), as has been done for Schwartz kernels of Toeplitz operators and of quantum propagators (\cite{charlesbohr, charles_floch}).
\subsubsection*{Acknowledgements}
This research has been supported by the DFG funded project SFB/TRR 191 (Project-ID 281071066-TRR 191) "Symplectic Structures in Geometry, Algebra and Dynamics", and the ANR-DFG project QuaSiDy (Project-ID ANR-21-CE40-0016) "Quantization, Singularities, and Holomorphic Dynamics" . I wish to express my sincere gratitude to the ANR and to the DFG.

I wish to thank X. Ma and G. Marinescu for many helpful discussions and for their valuable comments on an earlier version of this work. I also wish to thank L. Charles, C.-Y. Hsiao, N. Savale and A. Uribe for interesting and useful discussions on the topics addressed in this work.

\textsc{Ood Shabtai\\ \\
Université Paris Cité\\
CNRS, IMJ-PRG, B\^{a}timent Sophie Germain, UFR de Math\'{e}matiques
Case 7012, 75205 Paris CEDEX 13\\
France\\
}
\textit{E-mail address}: shabtai@imj-prg.fr\\ \\
\textsc{AND\\ \\}
\textsc{Université de Lille\\
Laboratoire de Mathématiques Paul Painlevé\\
CNRS U.M.R. 8524\\
59655 Villeneuve d'Ascq CEDEX\\
France}
\end{document}